\documentclass[12pt]{article}

\usepackage[a4paper,margin=31mm]{geometry}
\usepackage[T1]{fontenc}
\usepackage{amsmath,amssymb,amsthm,mathtools,mathrsfs}
\usepackage{microtype}
\usepackage{graphicx}
\usepackage{tikz}
\usetikzlibrary{arrows.meta,positioning}
\usepackage{float}
\usepackage[colorlinks=true,linkcolor=blue,citecolor=blue,urlcolor=blue]{hyperref}

\numberwithin{equation}{section}

\newtheorem{theorem}{Theorem}[section]
\newtheorem{proposition}[theorem]{Proposition}
\newtheorem{lemma}[theorem]{Lemma}

\theoremstyle{remark}
\newtheorem{remark}[theorem]{Remark}

\newcommand{\Area}{\operatorname{Area}}
\newcommand{\Det}{\operatorname{det}}
\newcommand{\Ind}{\operatorname{Ind}}
\newcommand{\Reg}{\operatorname{Reg}}
\newcommand{\CC}{\mathbb C}
\newcommand{\PP}{\mathbb P}
\newcommand{\cP}{\mathcal P}
\newcommand{\etaD}{\eta_{\mathrm D}}

\title{\bfseries Spectral determinants of flat metrics\\
on the bielliptic genus-two locus}
\author{Victor Kalvin}
\date{}

\begin{document}

\maketitle

\begin{abstract}
We obtain a closed explicit formula for the 
spectral determinant of flat conical metrics on the bielliptic genus-two
locus.  The metrics are generated by holomorphic one-forms with two
simple zeros.  For a symmetric reference metric, an exact Klein-four spectral
identity reduces the spectral determinant to scalar determinants on
spheres and tori; the singular anomaly formula then yields the general
case. The resulting
formula involves an elementary binary sextic in the
coefficients of the one-form and two explicit hypergeometric areas
of four-cone metrics on a quotient sphere.

We apply this formula to the separating, one-node nonseparating,
and simultaneous two-node degenerations of the curve and obtain
complete asymptotics of the spectral determinant in all cases.
The one-node degeneration has two distinct metric limits, according
to whether the limiting one-form is holomorphic or meromorphic.
Comparison of the cylindrical cases with the Bismut--Bost asymptotics
determines the corresponding constants explicitly; in the separating
case it also evaluates the relative determinant appearing in the
M\"uller--M\"uller formula.  As a by-product, the separating asymptotics
evaluate the multiplicative constants left undetermined in earlier
general conical and variational determinant formulas.
\end{abstract}

\section{Introduction}
\label{sec:introduction}
The purpose of this paper is to obtain a  closed
explicit formula for the spectral determinant over the full two-dimensional locus
of genus-two curves admitting an elliptic involution.

For the natural smooth hyperbolic metric on a general curve in this
locus, such an evaluation is obstructed by the absence of an explicit
uniformization.  For certain quasiplatonic surfaces, quotient relations
allow one to bypass this obstruction; this was carried out, for example, for the
Bolza surface and the Klein quartic in \cite{KalvinBolzaKlein}.

In this paper, we develop this quotient approach for the flat conical
metrics generated by holomorphic one-forms.  At the level of
uniformization, these metrics are fundamentally simpler: all relevant
data are explicit.  Together with an exact spectral reduction, this makes
it possible to carry out the determinant calculation completely over the
entire bielliptic locus.

We use the following two-parameter model of the bielliptic genus-two locus:
\begin{equation}\label{eq:curve}
X_{a,b}:\ y^2=(x^2-1)(x^2-a)(x^2-b),\quad  a,b\in\CC\setminus\{0,1\},\qquad a\ne b.
\end{equation}
The pair $(a,b)$ determines the complex structure of the compact Riemann surface $X_{a,b}$.  For
fixed $(a,b)$, consider the family
\begin{equation}\label{eq:omega-c}
\CC^2\setminus\{0\}\ni  \boldsymbol c=(c_0,c_1)\mapsto  \omega_{\boldsymbol c}
 =2(c_0+c_1x)\frac{dx}{y}.
 \end{equation}
As $\boldsymbol c$ varies over $\CC^2\setminus\{0\}$, the differential $\omega_{\boldsymbol c}$
runs through all nonzero holomorphic one-forms on $X_{a,b}$.  We assume
that $\omega_{\boldsymbol c}$ has two simple zeros.  The metric
$|\omega_{\boldsymbol c}|^2$ is flat away from these zeros and has at
each of them a conical singularity of order $1$, i.e. of angle $4\pi$.
The spectral determinant of the Friedrichs
Laplacian $\Delta_{|\omega_{\boldsymbol c}|^2}$ on
$(X_{a,b},|\omega_{\boldsymbol c}|^2)$ is defined in the standard way:
\[
 \Det\Delta_{|\omega_{\boldsymbol c}|^2}
 :=
 \exp\left\{
 -\left.\frac{d}{ds}\right|_{s=0}
 \sum_{\lambda>0}\lambda^{-s}
 \right\},
\]
where the eigenvalues are counted with multiplicity and the
sum is understood by analytic continuation.

The Klein-four reduction of  $X_{a,b}$ produces two double covers
of a sphere.  After M\"obius normalization, their branch sets are
$\{0,1,\lambda_j,\infty\}$, where
\begin{equation}\label{eq:lambda-coordinates}
 \lambda_1=\frac{b-1}{a-1},
 \qquad
 \lambda_2=\frac{a(1-b)}{a-b},
\end{equation}
respectively.  For either cover, the flat quotient metric on the base
sphere $\PP^1$ is
\begin{equation}\label{eq:legendre-base-metric}
 m_\lambda
 =|u(1-u)(\lambda-u)|^{-1}{|du|^2},\qquad \lambda\in\{\lambda_1,\lambda_2\}.
\end{equation}
 The metric $ m_\lambda$ has four conical singularities of order $-1/2$ (i.e. of angle
$\pi$) at $0$, $1$, $\lambda$, and $\infty$. 
The total area is 
\begin{equation}\label{eq:base-area-function}
 \Area(\PP^1,m_\lambda)
 =2\pi^2\Re\left(
 \overline{\lambda^{-1/2}F(\lambda^{-1})}
 F(1-\lambda)
 \right),\quad F(t):={}_2F_1(1/2,1/2;1;t).
\end{equation}
Here ${}_2F_1$ denotes the Gauss hypergeometric function.  For
$\lambda>1$, the square root and the two hypergeometric factors in
\eqref{eq:base-area-function} take their positive real values.  For
other $\lambda$, these factors are analytically continued from the
interval $(1,\infty)$; the resulting real quantity is single-valued.

Introduce the binary sextic
\begin{equation}\label{eq:binary-sextic}
 \cP_{a,b}(c_0,c_1)
 = (c_0^2-c_1^2)(c_0^2-ac_1^2)(c_0^2-bc_1^2). 
\end{equation}
The condition $\cP_{a,b}(c_0,c_1)\neq0$ is precisely the assumption
that $\omega_{\boldsymbol c}$ has two simple zeros.  Once the spectral determinant in the metric
$|\omega_{\boldsymbol c}|^2$ is known, the singular anomaly formula
\cite[Corollary~1.3]{KalvinJFA} extends the result to the general class of
conformal metrics with conical singularities considered there.

\begin{theorem}\label{thm:main}
Let $X_{a,b}$  be the surface in~\eqref{eq:curve} equipped with the flat metric $|\omega_{\boldsymbol c}|^2$. Suppose that
the binary sextic $\cP_{a,b}(c_0,c_1)$ in~\eqref{eq:binary-sextic} does not vanish.  Then
\begin{equation}\label{eq:main-formula}
 \begin{aligned}
 \frac{\Det\Delta_{|\omega_{\boldsymbol c}|^2}}
      {\Area(X_{a,b},|\omega_{\boldsymbol c}|^2)}
 &={}
 \frac{e^{6\zeta_R'(-1)}}{8\pi^3}
 |ab|^{-1/6}
 |\cP_{a,b}(c_0,c_1)|^{1/12}\\
 &\quad\times
 \prod_{j=1}^2
 \left(
 |\lambda_j(\lambda_j-1)|^{1/3}
 \Area(\PP^1,m_{\lambda_j})
 \right).
 \end{aligned}
\end{equation}
Here $\lambda_1$ and $\lambda_2$ are defined in~\eqref{eq:lambda-coordinates},
and $\zeta_R'(-1)$ denotes the derivative of the Riemann zeta function
at $-1$. Moreover, the area appearing on the left-hand side of
\eqref{eq:main-formula} satisfies
\begin{equation}\label{eq:general-area}
 \Area(X_{a,b},|\omega_{\boldsymbol c}|^2)
 =\frac{4|c_0|^2}{|a-b|}\Area(\PP^1,m_{\lambda_2})
 +\frac{4|c_1|^2}{|a-1|}\Area(\PP^1,m_{\lambda_1}),
\end{equation}
where $\Area(\PP^1,m_{\lambda})$ is given explicitly
by~\eqref{eq:base-area-function}.
\end{theorem}

To the best of our knowledge, this is the first absolutely normalized
closed explicit evaluation of the spectral determinant throughout a
two-dimensional genus-two locus.
Moreover, the final formula~\eqref{eq:main-formula} is remarkably simple: apart from its explicit
universal constant, it involves only algebraic functions of
$a,b,c_0,c_1$ and two explicit hypergeometric areas.  No choice of
symplectic bases or period ratios enters its statement. For comparison, the determinant of the smooth unit-area Bergman
metric on the same locus was computed in
\cite{KleinKokotovKorotkinBergman} only up to a moduli-independent
multiplicative constant, which remains undetermined.

The decisive step in the proof of Theorem~\ref{thm:main} is to
establish the Klein-four spectral identity for the Friedrichs
Laplacian of a conical reference metric, thereby
reducing its determinant to explicitly evaluable scalar quotient
determinants.  Without this reduction, our method would leave the
reference determinant unevaluated and would not produce a closed formula.

We  also apply Theorem~\ref{thm:main} to obtain complete asymptotics of the spectral
determinant under three degenerations of the curve $X_{a,b}$  In the
separating case the curve splits into two elliptic components.  In the
one-node nonseparating case, the normalization of the limiting nodal
curve is an elliptic curve, and the metric limit depends on the
differential: for $c_0=0$ its metric
completion is a smooth flat torus, whereas for $c_0\ne0$ it has two
cylindrical ends.  In the simultaneous two-node degeneration, the
normalization of the limiting nodal curve is a rational curve and the
metric limit has four cylindrical ends.  All the asymptotic formulas include their absolute leading
constants.

The degeneration formulas are then compared with the general
asymptotics of Bismut--Bost~\cite{BismutBost}.  In the separating case,
the M\"uller--M\"uller analytic-surgery formula
\cite[(7.39)]{MullerMuller} identifies the leading coefficient in the
general Bismut--Bost asymptotics with a relative determinant.  For the
family considered here, our calculation evaluates this coefficient, and
hence the relative determinant, explicitly.  The formulas for the
meromorphic one-node and simultaneous two-node degenerations likewise
determine their Bismut--Bost constants explicitly.

As a further by-product, the separating asymptotics determine the
multiplicative constants left undetermined in the degeneration formula
\cite{KokotovDegeneration} and in the variational determinant formula in
every genus~\cite{KokotovKorotkin}.  Finally, the same asymptotics
determine the universal constant in
\cite[Proposition~3.9]{KalvinJFA} and thus provide an absolutely
normalized determinant formula for all admissible conical metrics
considered there in every genus $g>1$; see
Remark~\ref{rem:jfa-normalization}.

The paper is organized as follows.  Section~\ref{sec:klein-four-quotient-geometry}
describes the Klein-four quotient geometry, and
Section~\ref{sec:artin-sunada-spectral-reduction} establishes the
Artin--Sunada spectral identity.  Sections~\ref{sec:spherical-quotient-determinants}
and~\ref{sec:elliptic-quotient-determinants} evaluate the spherical and
elliptic quotient determinants.  Section~\ref{sec:main-determinant} proves
Theorem~\ref{thm:main}, and Section~\ref{sec:degenerations} studies the
separating and nonseparating degenerations and the corresponding
asymptotics of the spectral determinant.

\section{The Klein-four quotient geometry
}
\label{sec:klein-four-quotient-geometry}

Consider the hyperelliptic involution $h$ and an elliptic involution $\sigma$, 
\begin{equation}\label{eq:v4-actions}
 h(x,y)=(x,-y),
 \qquad
 \sigma(x,y)=(-x,y),
\end{equation}
of the compact Riemann surface $X_{a,b}$ in~\eqref{eq:curve}. They commute and generate the Klein four-group
\[
 V=\{1,h,\sigma,h\sigma\}\cong C_2\times C_2
\]
acting on $X_{a,b}$.  The function $z=x^2$ is invariant under $V$.
The hyperelliptic projection $x:X_{a,b}\to\PP^1_x$ and the map
$x\mapsto x^2$ both have degree two, so $z$ has degree four.  Since
$V$ has order four, $z$ realizes the quotient map by $V$:
\begin{equation}\label{eq:full-quotient-map}
 q:X_{a,b}\longrightarrow S:=X_{a,b}/V\cong\PP^1_z,
 \qquad z=x^2.
\end{equation}
The branch values of $q$ are
\begin{equation}\label{eq:five-branch-values}
 0,\quad \infty,\quad 1,\quad a,\quad b.
\end{equation}
The first two come from the fixed points of $\sigma$ and
$h\sigma$, respectively, while the last three  come from the six fixed
points of $h$.  Thus the five marked points have two complex moduli,
as is also clear from \eqref{eq:curve}.

The quotients by the three subgroups of order two are
\begin{equation}\label{eq:intermediate-quotients}
 E_1=X_{a,b}/\langle\sigma\rangle,
 \qquad
 E_2=X_{a,b}/\langle h\sigma\rangle,
 \qquad
 S_h=X_{a,b}/\langle h\rangle.
\end{equation}
Let $\pi_j:X_{a,b}\to E_j$ denote the quotient maps.  The
$\sigma$-invariant functions $(z,Y)=(x^2,y)$ and the
$h\sigma$-invariant functions $(z,W)=(x^2,xy)$, respectively, put the
quotient curves into the forms
\begin{equation}\label{eq:elliptic-quotients-intro}
 \begin{split}
 E_1&:\quad Y^2=(z-1)(z-a)(z-b),\\
 E_2&:\quad W^2=z(z-1)(z-a)(z-b).
 \end{split}
\end{equation}
Thus $E_1$ and $E_2$ are elliptic curves, while the hyperelliptic
quotient $S_h$ is the $x$-sphere.

The forms
\begin{equation}\label{eq:canonical-differentials}
 \omega_{(1,0)}=2{dx}/{y},
 \qquad
 \omega_{(0,1)}=2x{dx}/{y}
\end{equation}
form a basis of the space of holomorphic one-forms on $X_{a,b}$;
see~\eqref{eq:omega-c}.  For every $g\in V$, the pullback $g^*$
preserves each of the two forms up to sign.
They are the pullbacks of the holomorphic differentials on the two
elliptic quotients:
\begin{equation}\label{eq:differential-pullbacks}
 \omega_{(1,0)}=\pi_2^*\frac{dz}{W},
 \qquad
 \omega_{(0,1)}=\pi_1^*\frac{dz}{Y},
\end{equation}
where $W$ and $Y$ are the same as in~\eqref{eq:elliptic-quotients-intro}.
The zero divisor of $\omega_{(1,0)}$ consists of the two points at
infinity, whereas that of $\omega_{(0,1)}$ consists of the two points
over $x=0$.

For $\lambda>1$ consider the model elliptic curve $$E_\lambda:v^2=u(1-u)(\lambda-u)$$ and the holomorphic differential 
$w_\lambda={du}/{v}$.  To fix the sign of $w_\lambda$
and the lifts of the integration intervals, choose $v>0$ for
$u\in(0,1)$ and continue $v$ through the upper half-plane across
$u=1$.  Then, for $u\in(1,\lambda)$,
\[
 v=-i\sqrt{u(u-1)(\lambda-u)},
\]
where the square root on the right is positive.  With these
conventions, the integrals
\begin{equation}\label{eq:base-edge-integrals}
 I_0(\lambda)=\int_0^1w_\lambda,
 \qquad
 I_1(\lambda)=\int_1^\lambda w_\lambda
\end{equation}
are half-periods of  $w_\lambda$ on
$E_\lambda$, so that $I_0(\lambda)>0$,
$ i I_1(\lambda)<0$, and
$\Im\bigl(\overline{I_0(\lambda)}I_1(\lambda)\bigr)>0$.
Euler's integral representation of the Gauss hypergeometric function,
\[
 {}_2F_1\left(\frac12,\frac12;1;t\right)
 =\frac1\pi\int_0^1
 \frac{ds}{\sqrt{s(1-s)(1-ts)}},
 \qquad t<1,
\]
applies directly to the half-periods in
\eqref{eq:base-edge-integrals}.  Factoring $\sqrt\lambda$ out of the
denominator in $I_0$ and setting $u=1+(\lambda-1)s$ in $I_1$ yields
\[
 I_0(\lambda)
 =\frac{\pi}{\sqrt\lambda}\,
 {}_2F_1\left(\frac12,\frac12;1;\lambda^{-1}\right),\qquad
 I_1(\lambda)
 =i\pi\,
 {}_2F_1\left(\frac12,\frac12;1;1-\lambda\right),
\]
where $\sqrt\lambda>0$.  Since $\lambda^{-1}\in(0,1)$ and
$1-\lambda<0$, both hypergeometric functions are represented by the
preceding integral.  For general
$\lambda\in\CC\setminus\{0,1\}$, the differential $w_\lambda$ and the
period pair $(I_0,I_1)$ are defined by analytic continuation of the
preceding expressions from the interval $(1,+\infty)$.  Although the
period pair is generally multivalued, its monodromy is symplectic and
therefore preserves $2\Im(\overline{I_0}I_1)$.

Under the involution $(u,v)\mapsto(u,-v)$ of the double cover
$E_\lambda\to\mathbb P^1$, the differential $w_\lambda$ changes sign,
whereas the metric $|w_\lambda|^2$ is invariant and therefore descends
to the flat four-cone metric $m_\lambda$
in~\eqref{eq:legendre-base-metric} on $\mathbb P^1$.
The second Riemann bilinear relation on $E_\lambda$ gives
\[
 \Area(\PP^1,m_\lambda)
 =2\Im\bigl(\overline{I_0(\lambda)}I_1(\lambda)\bigr).
\]
Together with the hypergeometric expressions for the two half-periods,
this implies~\eqref{eq:base-area-function}.

Next, we relate the model elliptic curve $E_\lambda$ to  the elliptic curves $E_j$ in \eqref{eq:elliptic-quotients-intro}. 
 On the quotient sphere $S$ from~\eqref{eq:full-quotient-map}, set
\begin{equation}\label{eq:two-base-metrics}
m_S^{(0,1)}
 =|(z-1)(z-a)(z-b)|^{-1}|dz|^2 ,
 \quad
m_S^{(1,0)}
 =|z(z-1)(z-a)(z-b)|^{-1}|dz|^2.
\end{equation}
The substitutions
\[
 u=\frac{z-1}{a-1}\quad\text{on }E_1,
 \qquad
 u=\frac{(1-b)z}{z-b}\quad\text{on }E_2
\]
transform the two metrics according to
\begin{equation}\label{eq:base-legendre-pullbacks}
 m_S^{(0,1)}=\frac1{|a-1|}m_{\lambda_1},
 \qquad
 m_S^{(1,0)}=\frac1{|a-b|}m_{\lambda_2},
\end{equation}
where $\lambda_1$ and $\lambda_2$ are defined
in~\eqref{eq:lambda-coordinates} and $m_\lambda$ is the flat four-cone metric in~\eqref{eq:legendre-base-metric}.
The period ratios of the elliptic curves $E_j$ in \eqref{eq:elliptic-quotients-intro} are therefore
\[
 \tau_j=I_1(\lambda_j)/{I_0(\lambda_j)},\quad \Im  \tau_j>0,
 \qquad j=1,2.
\]
The degree-two covers
$E_1\to(S,m_S^{(0,1)})$ and  $E_2\to(S,m_S^{(1,0)})$, together with
\eqref{eq:base-legendre-pullbacks}, imply
\begin{equation}\label{eq:explicit-base-areas}
 \begin{aligned}
  \Area(S,m_S^{(0,1)})
 &=\frac12\Area(E_1,|dz/Y|^2)
 =\frac{\Area(\PP^1,m_{\lambda_1})}{|a-1|},\\
 \Area(S,m_S^{(1,0)})
 &=\frac12\Area(E_2,|dz/W|^2)
 =\frac{\Area(\PP^1,m_{\lambda_2})}{|a-b|}.
 \end{aligned}
\end{equation}
Finally, summing the area density of
$\omega_{\boldsymbol c}=c_0\omega_{(1,0)}+c_1\omega_{(0,1)}$ over the four
sheets of $q:X_{a,b}\to S$ cancels the mixed terms.  Using
\eqref{eq:explicit-base-areas}, we obtain
\[
 \begin{aligned}
 \Area(X_{a,b},|\omega_{\boldsymbol c}|^2)
 ={}&4|c_0|^2\Area(S,m_S^{(1,0)})
 +4|c_1|^2\Area(S,m_S^{(0,1)})\\
 ={}&2|c_0|^2\Area(E_2,|dz/W|^2)
 +2|c_1|^2\Area(E_1,|dz/Y|^2)\\
 ={}&\frac{4|c_0|^2}{|a-b|}\Area(\PP^1,m_{\lambda_2})
 +\frac{4|c_1|^2}{|a-1|}\Area(\PP^1,m_{\lambda_1}).
 \end{aligned}
\]
This is exactly the result stated in \eqref{eq:general-area}. The previously deduced
formula~\eqref{eq:base-area-function} for the area of $(\PP^1,m_\lambda)$ expresses $ \Area(X_{a,b},|\omega_{\boldsymbol c}|^2)$  in terms of  hypergeometric
functions.

\section{Artin--Sunada spectral reduction}
\label{sec:artin-sunada-spectral-reduction}

The  flat metric $|\omega_{\boldsymbol c}|^2$ is not necessarily
$V$-invariant.  Indeed, when $c_0c_1\ne0$, neither elliptic
involution preserves the metric, and the spectral quotient identity
below does not apply.  We therefore use $|\omega_{(1,0)}|^2$ as the
reference metric.  The second $V$-eigenform $\omega_{(0,1)}$ could
equally be used, but one symmetric reference value suffices: the
determinant for arbitrary $\omega_{\boldsymbol c}$ then follows from
the singular anomaly formula
\cite[Corollary~1.3]{KalvinJFA}.

The underlying $V_4$ permutation-representation identity is an
Artin--Sunada relation \cite{Sunada,Parzanchevski}; for smooth metrics its spectral consequence is
\cite[Corollary~3.3]{Parzanchevski}.  For the singular metric $|\omega_{(1,0)}|^2$
 the corresponding
consequence for the Friedrichs Laplacians is verified below.  In general, the
representation-theoretic identity alone does not guarantee an explicit
evaluation of the spectral determinant: higher-genus quotients or twisted determinants may remain.
The key point here is that the reduction closes entirely on scalar
Laplacians on the spherical and elliptic quotients of $X_{a,b}$.  All
these quotient determinants are evaluated explicitly in
Sections~\ref{sec:spherical-quotient-determinants}
and~\ref{sec:elliptic-quotient-determinants}.

\begin{proposition}[Artin-Sunada reduction]\label{prop:v4-flat-reduction} Equip  the quotient sphere $S$ in~\eqref{eq:full-quotient-map} with the metric 
$m_S^{(1,0)}$ in~\eqref{eq:two-base-metrics}.
Then the pullback metric on $X_{a,b}$ is
$
 q^*m_S^{(1,0)}
 =|\omega_{(1,0)}|^2$.
The three intermediate quotients  $E_1, E_2,S_h$ are then naturally equipped with the following metrics:
\begin{equation}\label{eq:quotient-metrics}
 \begin{aligned}
 m_1&=\frac1{|z|}\left|\frac{dz}{Y}\right|^2
       &&\text{on }E_1,\\
 m_2&=\left|\frac{dz}{W}\right|^2
       &&\text{on }E_2,\\
 m_h&=\frac{4|dx|^2}
 {|(x^2-1)(x^2-a)(x^2-b)|}
       &&\text{on }S_h.
 \end{aligned}
\end{equation}
 Let $\Delta_{E_j}$, $\Delta_{S_h}$, and $\Delta_S$ stand for the Friedrichs Laplacians on the  surfaces $(E_j, m_j)$, $(S_h,m_h)$, and $(S, m_S^{(1,0)})$ respectively. 
Then their spectral determinants satisfy the relation
\begin{equation}\label{eq:v4-flat-reduction}
 \Det\Delta_{|\omega_{(1,0)}|^2}
 =\frac{
  \Det\Delta_{E_1}\,
  \Det\Delta_{E_2}\,
  \Det\Delta_{S_h}
 }{
  (\Det\Delta_{S})^2
 }.
\end{equation}
\end{proposition}

\begin{remark}
Note that  the metric $m_S^{(1,0)}$ on $S$ has conical singularities of order
$-1/2$ at $z\in\{0,1,a,b\}$, i.e. four cone angles $\pi$.
Among the intermediate quotients, $(E_2,m_2)$ is a smooth flat torus.  The metric $m_1$ on $E_1$ has conical singularities of order
$-1/2$ at the two points over $0\in\PP^1_z$ and a conical
singularity of order $1$ at the unique point over
$\infty\in\PP^1_z$, i.e. two
cone angles $\pi$ and one cone angle $4\pi$.  Finally, the metric
$m_h$ on the sphere $S_h$ has conical singularities of order $-1/2$ at the six
points $x\in\{\pm1,\pm\sqrt a,\pm\sqrt b\}$ and of order $1$ at
$x=\infty$, i.e. six cone angles $\pi$ and one cone angle $4\pi$.
\end{remark}

\begin{proof} [Proof of Proposition~\ref{prop:v4-flat-reduction}] For a subgroup $H\leq V$, let
\[
 \pi_H:X_{a,b}\longrightarrow X_{a,b}/H
\]
be the quotient map and let $m_H$ be the quotient metric,
characterized away from the ramification points by
 $\pi_H^*m_H=|\omega_{(1,0)}|^2$.
For every smooth function $v$ supported away from the conical and branch  values,
\[
 \|\pi_H^*v\|_{L^2(X_{a,b},|\omega_{(1,0)}|^2)}^2
 =\deg\pi_H\,\|v\|_{L^2(X_{a,b}/H,m_H)}^2
\]
and
\[
 \mathcal Q_{X_{a,b}}(\pi_H^*v)=\deg\pi_H\,\mathcal Q_H(v),
\]
where $\mathcal Q_{X_{a,b}}$ and $\mathcal Q_H$ are the corresponding quadratic (energy)
forms.  The omitted points have zero $H^1$-capacity, so these test
functions form a core for the corresponding Friedrichs forms.
Passing to the closures shows that
\[
 v\longmapsto (\deg\pi_H)^{-1/2}\pi_H^*v
\]
is a unitary map from the domain  of the quadratic form $\mathcal Q_H$ on $X_{a,b}/H$ onto
the $H$-invariant part of the domain of $\mathcal Q_{X_{a,b}}$  on $X_{a,b}$.


Now the elementary identity of $V$-representations
\begin{equation}\label{eq:v4-representation}
 \Ind_{\langle h\rangle}^{V}{\bf1}
 +\Ind_{\langle\sigma\rangle}^{V}{\bf1}
 +\Ind_{\langle h\sigma\rangle}^{V}{\bf1}
 =\Reg_V+2{\bf1}_V
\end{equation}
together with Frobenius reciprocity implies, for every $\lambda>0$,
\begin{equation}\label{eq:v4-eigenvalue-multiplicities}
 \begin{aligned}
 \dim\ker(\Delta_{|\omega_{(1,0)}|^2}-\lambda)
 ={}&\dim\ker(\Delta_{S_h}-\lambda)
 +\dim\ker(\Delta_{E_1}-\lambda)\\
 &+\dim\ker(\Delta_{E_2}-\lambda)
 -2\dim\ker(\Delta_{S}-\lambda).
 \end{aligned}
\end{equation}
Consequently,  for the spectral zeta functions we have
\[
 \zeta_{|\omega_{(1,0)}|^2}(s)
 =\zeta_{S_h}(s)+\zeta_{E_1}(s)+\zeta_{E_2}(s)
 -2\zeta_{S}(s),
\]
initially for $\Re s>1$ and then by meromorphic continuation.
Differentiating the latter equality at $s=0$ gives
\eqref{eq:v4-flat-reduction}.
\end{proof}


\section{Determinants of the two spherical quotients}
\label{sec:spherical-quotient-determinants}
The general singular anomaly formula~\cite{KalvinJFA} was  already specialized to the case of any flat metric on $\PP^1$ in~\cite[Proposition~3.3]{KalvinJFA}. Moreover, for four
singularities of order $-1/2$ it was subsequently written in unit-area
form in \cite[Section~4]{KalvinAnnali}. In the present argument the same
four-cone metric occurs intrinsically as the  quotient metric $m_S^{(1,0)}$ in the
Klein-four reduction.  It is straightforward to implement those results from~\cite{KalvinJFA,KalvinAnnali} here and obtain the required explicit expressions for the determinants $\Det\Delta_{S}$ and $\Det\Delta_{S_h}$ on the flat conical spheres. Therefore we only state the result in Lemma~\ref{Lmm} below and omit the details.

\begin{lemma}[Determinants of the two spherical quotients]\label{Lmm}
\leavevmode
\begin{enumerate}
\item
For the determinant of the Friedrichs Laplacian on the sphere $S$, equipped with the metric $m_S^{(1,0)}$  having conical singularities of order $-1/2$ at the points $z\in\{0,1,a,b\}$,
 one has
\begin{equation}\label{eq:base-sphere-determinant}
\begin{aligned}
 \log\Det\Delta_S
 ={}&\log\Area(S,m_S^{(1,0)})
 +\frac16\log|ab(1-a)(1-b)(a-b)|\\
 &-\frac23\log2-\log\pi.
\end{aligned}
\end{equation}
\item For the determinant of the Friedrichs Laplacian on the sphere $S_h$, equipped with the metric $m_h$ having conical singularities of order $-1/2$ at the six points $x\in\{\pm 1,\pm \sqrt a,\pm \sqrt b\}$  and of order $1$ at infinity, 
one has 
\begin{equation}\label{eq:hyperelliptic-sphere-determinant}
\begin{aligned}
 \log\Det\Delta_{S_h}
 ={}&\log\Area(S_h,m_h)
 +\frac13\log|(1-a)(1-b)(a-b)|\\
 &+\frac1{12}\log|ab|
 +3\zeta_R'(-1)-\log(2\pi).
\end{aligned}
\end{equation}
\end{enumerate}
The areas in these formulas are related to the area explicitly  evaluated in
\eqref{eq:base-area-function} by
\[
 \Area(S,m_S^{(1,0)})
 =\frac{\Area(\PP^1,m_{\lambda_2})}{|a-b|},
 \qquad
 \Area(S_h,m_h)
 =\frac{2\Area(\PP^1,m_{\lambda_2})}{|a-b|}.
\]
\end{lemma}


\section{Determinants of the two elliptic quotients}
\label{sec:elliptic-quotient-determinants}

As we show in the proof of Lemma~\ref{lem:elliptic-determinants} below,
the determinant of the smooth quotient $(E_2,m_2)$ follows directly
from the Kronecker limit formula, whereas the singular anomaly formula
reduces the determinant of $(E_1,m_1)$ to that of a smooth flat torus.
Classical theta identities then eliminate the resulting modular factors.

\begin{lemma}[Determinants of the two elliptic quotients]
\label{lem:elliptic-determinants}
The determinants of the Friedrichs Laplacians on $(E_2,m_2)$ and
$(E_1,m_1)$ are
\begin{equation}\label{eq:e2-determinant}
 \Det\Delta_{E_2}
 =\frac{|\lambda_2(\lambda_2-1)|^{1/3}}
 {2^{4/3}\pi^2|a-b|}
\left( \Area(\PP^1,m_{\lambda_2})\right)^2
\end{equation}
and
\begin{equation}\label{eq:e1-anomaly}
 \Det\Delta_{E_1}
 =\frac{e^{3\zeta_R'(-1)}|ab|^{1/12}}
 {2\pi^2|a-b|}
 |\lambda_1(\lambda_1-1)|^{1/3}
 \prod_{j=1}^2\Area(\PP^1,m_{\lambda_j}).
\end{equation}
Here $\lambda_j$ are the same as in~\eqref{eq:lambda-coordinates} and the areas are explicitly expressed in~\eqref{eq:base-area-function}.
\end{lemma}

\begin{proof}
We first  assume that  $\lambda>1$ and set
$t=\lambda^{-1}\in(0,1)$.  With the same
$F(t)={}_2F_1(1/2,1/2;1;t)$ as in~\eqref{eq:base-area-function}, the
period conventions fixed above yield
\[
 \tau(\lambda)=\frac{I_1(\lambda)}{I_0(\lambda)}
 =i\frac{F(1-t)}{F(t)}.
\]
The classical identities
\[
 t=\frac{\vartheta_2(\tau)^4}{\vartheta_3(\tau)^4},
 \qquad
 1-t=\frac{\vartheta_4(\tau)^4}{\vartheta_3(\tau)^4},
 \qquad
 F(t)=\vartheta_3(\tau)^2,
 \qquad
 2\etaD(\tau)^3
 =\vartheta_2(\tau)\vartheta_3(\tau)\vartheta_4(\tau),
\]
 where $\etaD$ is the Dedekind eta function, imply
\[
 |\etaD(\tau(\lambda))|^4
 =2^{-4/3}|t(1-t)|^{1/3}|F(t)|^2.
\]
Since
\[
 I_0(\lambda)=\pi\lambda^{-1/2}F(t),
 \qquad
 \Area(\PP^1,m_\lambda)
 =2|I_0(\lambda)|^2\Im\tau(\lambda),
\]
we obtain
\begin{equation}\label{eq:kronecker-lambda-factor}
 \Im\tau(\lambda)|\etaD(\tau(\lambda))|^4
 =\frac{|\lambda(\lambda-1)|^{1/3}}{2^{7/3}\pi^2}
 \Area(\PP^1,m_\lambda).
\end{equation}
This establishes~\eqref{eq:kronecker-lambda-factor} for $\lambda>1$.
 Since
$\Im\tau\,|\etaD(\tau)|^4$ is invariant under symplectic changes of
the period basis, \eqref{eq:kronecker-lambda-factor} extends as a
single-valued function to every $\lambda\in\CC\setminus\{0,1\}$.

The metric $m_2=|dz/W|^2$ is smooth and flat.  The Kronecker limit
formula \cite{OsgoodPhillipsSarnak} yields
\[
 \Det\Delta_{E_2}
 =\Area(E_2,m_2)\Im\tau_2\,|\etaD(\tau_2)|^4.
\]
This together with \eqref{eq:explicit-base-areas} and
\eqref{eq:kronecker-lambda-factor} proves
\eqref{eq:e2-determinant}.

To evaluate the second determinant, we use the singular anomaly formula~\cite[Theorem 1.1]{KalvinJFA} for the metric $m_1=|z|^{-1} \left|{dz}/{Y}\right|^2$ and  the smooth flat
metric $\left|{dz}/{Y}\right|^2$.

Let $P_+$ and $P_-$ be the two points over $z=0$, and let $P_\infty$
be the point at infinity.  In a local coordinate $x$ centred at one of
these points, write
\[
\begin{aligned}
 m_1 & =|x|^{2\beta_P}e^{2\phi}|dx|^2  =|x|^{2\beta_P}e^{2\phi_P(0)+o(1)}|dx|^2, \quad |x|\to 0,
 \\
 \left|{dz}/{Y}\right|^2&=e^{2\psi}|dx|^2
 =e^{2\psi_P(0)+o(1)}|dx|^2,
 \qquad |x|\to 0.
 \end{aligned}
\]
Both metrics are flat wherever smooth, so the curvature integrals in
the singular anomaly formula~\cite[Theorem~1.1]{KalvinJFA} vanish.  As a result, the anomaly formula simplifies to 
\begin{equation}\label{eq:e1-anomaly-specialized}
 \begin{aligned}
 \log \frac{\Det\Delta_{E_1}/\Area(E_1,m_1)}{\Det\Delta_{\left|{dz}/{Y}\right|^2}/\Area(E_1,\left|{dz}/{Y}\right|^2)}
 =  & \frac16\sum_{P\in\{P_+,P_-,P_\infty\}}
 \beta_P\left(\frac{\phi_P(0)}{\beta_P+1}-\psi_P(0)\right)
 \\ & -\sum_{P\in\{P_+,P_-,P_\infty\}}C(\beta_P).
 \end{aligned}
\end{equation}

Since the projection $z:E_1\to\PP^1$ is unramified at $P_\pm$, we may
take $x=z$ as a local coordinate centred at either point; the identity
$Y^2(0)=-ab$ then implies
\[
 \beta_P=-\frac12,
 \qquad
 \phi_P(0)=\psi_P(0)=-\frac12\log|ab|.
\]

In a vicinity of $P_\infty$ we use the local coordinate $x=z^{-1/2}$.  Then
\[
 \frac{dz}{Y}=-2(1+O(x^2))\,dx,
 \qquad
 |z|^{-1}=|x|^2,
\]
and hence
\[
 \beta_{P_\infty}=1,
 \qquad
 \phi_{P_\infty}(0)=\psi_{P_\infty}(0)=\log2.
\]
Consequently, the finite points contribute
$\frac1{12}\log|ab|$ to the first sum in
\eqref{eq:e1-anomaly-specialized}, while the point at infinity
contributes $-\frac1{12}\log2$.  

The last sum in
\eqref{eq:e1-anomaly-specialized} is the contribution to the anomaly
that depends only on the orders of the conical singularities.  Its
required values are given explicitly in
\cite[p.~6, after Eq.~(1.5)]{KalvinJFA}:
\begin{equation}\label{eq:C-two}
 C\left(-\frac12\right)
 =-\zeta_R'(-1)-\frac16\log2+\frac1{24},
 \qquad
 C(1)=-\zeta_R'(-1)-\frac1{12}\log2-\frac1{12}.
\end{equation}
In summary, the anomaly formula \eqref{eq:e1-anomaly-specialized} reduces to
\[
  \log \frac{\Det\Delta_{E_1}/\Area(E_1,m_1)}{\Det\Delta_{\left|{dz}/{Y}\right|^2}/\Area(E_1,\left|{dz}/{Y}\right|^2)}
 = \frac1{12}\log|ab|+\frac13\log2+3\zeta_R'(-1).
\]
Finally, for the smooth flat metric $\left|{dz}/{Y}\right|^2$, the
Kronecker limit formula reads
\[
 \Det\Delta_{\left|{dz}/{Y}\right|^2}
 =\Area(E_1,\left|{dz}/{Y}\right|^2)\Im\tau_1|\etaD(\tau_1)|^4.
\]
Consequently,
\begin{equation}\label{eq:e1-modular-form}
 \Det\Delta_{E_1}
 =2^{1/3}e^{3\zeta_R'(-1)}|ab|^{1/12}
 \Area(E_1,m_1)\Im\tau_1|\etaD(\tau_1)|^4.
\end{equation}

Both quotient maps $X_{a,b}\to E_j$ have degree two and the quotient
metrics pull back to $|\omega_{(1,0)}|^2$.  Hence
\[
 \Area(E_1,m_1)=\Area(E_2,m_2)
 =\frac{2}{|a-b|}\Area(\PP^1,m_{\lambda_2}).
\]
Substitution of this identity into~\eqref{eq:e1-modular-form} and
\eqref{eq:kronecker-lambda-factor} with $\lambda=\lambda_1$ proves
\eqref{eq:e1-anomaly}.
\end{proof}


\section[The spectral determinant on the genus-two surface]{The spectral determinant of \((X_{a,b},|\omega_{\boldsymbol c}|^2)\)}
\label{sec:main-determinant}

 Substituting
the four quotient determinants evaluated above into the Artin--Sunada
identity~\eqref{eq:v4-flat-reduction} we find the determinant for the
reference metric $|\omega_{(1,0)}|^2$.  As we show in the proof of Theorem~\ref{thm:main} below, the singular anomaly formula
then determines the determinant for any metric
$|\omega_{\boldsymbol c}|^2$.

\begin{proof}[Proof of Theorem~\ref{thm:main}]

For the reference differential $\omega_{(1,0)}=2dx/y$, we first show
that
\begin{equation}\label{eq:omega-zero-formula}
 \frac{\Det\Delta_{|\omega_{(1,0)}|^2}}
 {\Area(X_{a,b},|\omega_{(1,0)}|^2)}
 =\frac{e^{6\zeta_R'(-1)}}{8\pi^3}|ab|^{-1/6}
 \prod_{j=1}^2
 \left(
 |\lambda_j(\lambda_j-1)|^{1/3}
 \Area(\PP^1,m_{\lambda_j})
 \right).
\end{equation}
Here $\lambda_j$ are defined in~\eqref{eq:lambda-coordinates}, and the
areas $\Area(\PP^1,m_{\lambda_j})$ are given explicitly
by~\eqref{eq:base-area-function}.

Insert the expressions obtained for the  four quotient determinants in \eqref{eq:base-sphere-determinant},
\eqref{eq:hyperelliptic-sphere-determinant},
\eqref{eq:e2-determinant}, and \eqref{eq:e1-anomaly}  into the 
Artin--Sunada identity~\eqref{eq:v4-flat-reduction}.  The algebraic position factors cancel
according to
\begin{equation}\label{eq:position-cancellation}
 \frac{
 |(1-a)(1-b)(a-b)|^{1/3}|ab|^{1/6}}
 {|ab(1-a)(1-b)(a-b)|^{1/3}}
 =|ab|^{-1/6}.
\end{equation}
Using the area relations in Lemmas~\ref{Lmm}
and~\ref{lem:elliptic-determinants}, the remaining factors combine to
produce
\begin{equation}\label{eq:aux}
 \begin{aligned}
 \Det\Delta_{|\omega_{(1,0)}|^2}
 ={}&\frac{e^{6\zeta_R'(-1)}}{2\pi^3|a-b|}
 |ab|^{-1/6}\Area(\PP^1,m_{\lambda_2})\\
 &\times\prod_{j=1}^2
 \left(
 |\lambda_j(\lambda_j-1)|^{1/3}
 \Area(\PP^1,m_{\lambda_j})
 \right).
 \end{aligned}
\end{equation}
Finally, \eqref{eq:general-area} with $(c_0,c_1)=(1,0)$ specializes to
\[
 \Area(X_{a,b},|\omega_{(1,0)}|^2)
 =\frac4{|a-b|}\Area(\PP^1,m_{\lambda_2}).
\]
This together with~\eqref{eq:aux}  implies~\eqref{eq:omega-zero-formula}.

Now we pass from $\omega_{(1,0)}$ to an arbitrary holomorphic
one-form with two simple zeros.    If $\omega$ and
$\widetilde\omega$ have simple zero divisors $\sum 1\cdot P_k$ and
$\sum1\cdot \widetilde P_k$ with disjoint supports, then the anomaly formula from~\cite[Corollary~1.3]{KalvinJFA} simplifies to
\begin{equation}\label{eq:abelian-comparison}
 \frac{
  \Det\Delta_{|\omega|^2}/\Area(X,|\omega|^2)}{
  \Det\Delta_{|\widetilde\omega|^2}/
  \Area(X,|\widetilde\omega|^2)}
 =
 \left|
 \frac{
  \prod_k\operatorname*{res}_{\widetilde P_k}
       (\omega^2/\widetilde\omega)
 }{
  \prod_k\operatorname*{res}_{P_k}
       (\widetilde\omega^2/\omega)
 }
 \right|^{1/12}.
\end{equation}

Set
\[
 f_{a,b}(x)=(x^2-1)(x^2-a)(x^2-b).
\]
  Let $Q_+$ and
$Q_-$ be the two points over $x=-c_0/c_1$, and denote the two points at
infinity by $\infty_+$ and $\infty_-$.  Direct calculation gives
\begin{equation}\label{eq:residues-infinity}
 \left|\operatorname*{res}_{\infty_\pm}
 \frac{\omega_{\boldsymbol c}^{2}}{\omega_{(1,0)}}\right|
 =2|c_1|^2,
\end{equation}
and
\begin{equation}\label{eq:residues-finite}
 \left|\operatorname*{res}_{Q_\pm}
 \frac{\omega_{(1,0)}^{2}}{\omega_{\boldsymbol c}}\right|
 =\frac2{|c_1|\,|f_{a,b}(-c_0/c_1)|^{1/2}},
\end{equation}
where we first assume that $c_1\ne0$. 
Therefore
\[
 \left|
 \frac{
  \prod_{\nu=\pm}\operatorname*{res}_{\infty_\nu}
  (\omega_{\boldsymbol c}^{2}/\omega_{(1,0)})}
 {
  \prod_{\nu=\pm}\operatorname*{res}_{Q_\nu}
  (\omega_{(1,0)}^{2}/\omega_{\boldsymbol c})}
 \right|
 =|c_1|^6|f_{a,b}(-c_0/c_1)|
 =|\cP_{a,b}(c_0,c_1)|.
\]
As a result, the equality \eqref{eq:abelian-comparison} implies
\begin{equation}\label{eq:general-differential-ratio}
 \frac{\Det\Delta_{|\omega_{\boldsymbol c}|^2}}
 {\Area(X_{a,b},|\omega_{\boldsymbol c}|^2)}
 =
 \frac{\Det\Delta_{|\omega_{(1,0)}|^2}}
 {\Area(X_{a,b},|\omega_{(1,0)}|^2)}
 |\cP_{a,b}(c_0,c_1)|^{1/12}.
\end{equation}
For $c_1=0$, the same identity follows directly from scaling.
Combining \eqref{eq:general-differential-ratio} with
\eqref{eq:omega-zero-formula} we complete the proof of  Theorem~\ref{thm:main}.
\end{proof}

\begin{remark}[Scaling check]\label{rem:scaling}
By~\cite[Corollary~1.2]{KalvinJFA},
$\zeta_{|\omega_{\boldsymbol c}|^2}(0)=-5/4$.  Hence the standard
rescaling relation~\cite[Eq.~(1.6)]{KalvinJFA} reads
\[
 \Det\Delta_{|\mu\omega_{\boldsymbol c}|^2}
 =|\mu|^{5/2}\Det\Delta_{|\omega_{\boldsymbol c}|^2},
 \qquad
 \Area(X_{a,b},|\mu\omega_{\boldsymbol c}|^2)
 =|\mu|^2\Area(X_{a,b},|\omega_{\boldsymbol c}|^2),
\]
where 
$\mu\in\CC\setminus\{0\}$ is a scaling coefficient.
Therefore
\[
 \frac{\Det\Delta_{|\mu\omega_{\boldsymbol c}|^2}}
      {\Area(X_{a,b},|\mu\omega_{\boldsymbol c}|^2)}
 =|\mu|^{1/2}
 \frac{\Det\Delta_{|\omega_{\boldsymbol c}|^2}}
      {\Area(X_{a,b},|\omega_{\boldsymbol c}|^2)}.
\]
Since $\mu\omega_{\boldsymbol c}=\omega_{\mu\boldsymbol c}$ and
$\cP_{a,b}$ is homogeneous of degree six, the right-hand side of
\eqref{eq:main-formula} has the same scaling.
\end{remark}

\begin{remark}[Independence of the normalized coordinate]
Although the closed epxplicit formula~\eqref{eq:main-formula} for 
the spectral determinant  of $(X_{a,b},|\omega_{\boldsymbol c}|^2)$  is written in the normalized
coordinate $x$, its right-hand side is independent of this choice.

Indeed, up to changes of signs of $x$ and $y$, every change of normalized
coordinate preserving the chosen bielliptic involution is generated by
$a\leftrightarrow b$ and
\[
 (a,b;c_0,c_1)
 \longmapsto
 \left(a^{-1},\frac ba;\frac{c_0}{a},\frac{c_1}{\sqrt a}\right),
\]
\[
 (a,b;c_0,c_1)
 \longmapsto
 \left(a^{-1},b^{-1};
 -\frac{c_1}{\sqrt{-ab}},-\frac{c_0}{\sqrt{-ab}}\right).
\]
The latter transformations correspond respectively to
$(x,y)=(\sqrt a\,\widetilde x,a^{3/2}\widetilde y)$ and
$(x,y)=(\widetilde x^{-1},\sqrt{-ab}\,\widetilde x^{-3}\widetilde y)$.
Directly,
\[
 \cP_{a^{-1},\,b/a}
 \left(\frac{c_0}{a},\frac{c_1}{\sqrt a}\right)
 =a^{-6}\cP_{a,b}(c_0,c_1),
\]
and
\[
 \cP_{a^{-1},\,b^{-1}}
 \left(-\frac{c_1}{\sqrt{-ab}},
       -\frac{c_0}{\sqrt{-ab}}\right)
 =(ab)^{-4}\cP_{a,b}(c_0,c_1).
\]
Hence
\[
 |ab|^{-1/6}|\cP_{a,b}(c_0,c_1)|^{1/12}
\]
is unchanged.  By~\eqref{eq:kronecker-lambda-factor}, the remaining
product in~\eqref{eq:main-formula} is a universal constant times the
product of the two Kronecker factors.  The first coordinate change
induces isomorphisms of the two elliptic quotients separately, whereas
inversion interchanges them.  Thus this product is unchanged as well;
all choices of signs and square roots disappear after taking absolute
values.  Hence
\eqref{eq:main-formula} is independent of the normalized coordinate.
\end{remark}

\section{Degenerations}
\label{sec:degenerations}

We now apply Theorem~\ref{thm:main} at the boundary of the
bielliptic locus. 
Algebraically, the curve acquires one or two nodes and is described by
its normalization.  Metrically, the limit depends on the behavior of
the differential at the preimages of the nodes: a holomorphic limit
leaves these points at finite distance, whereas simple poles with
opposite residues produce cylindrical ends.

Up to the changes of normalized coordinate discussed above, the
boundary regimes considered below are the separating degeneration,
the two metric realizations of the one-node nonseparating degeneration,
and the simultaneous two-node degeneration.

\subsection{The separating degeneration}
\label{sec:separating-degeneration}

Fix $A,B\in\CC\setminus\{0\}$, with $A\ne B$, and let
\begin{equation}\label{eq:abc}
 a=1+\varepsilon A,
 \qquad
 b=1+\varepsilon B.
\end{equation}
As $\varepsilon\to 0$ the three branch points near $1$ collide, and so do their images  under  the involution \(x\mapsto-x\)  near
$-1$.  After taking the quotient by $x\mapsto-x$ and introducing
\begin{equation}\label{eq:degen-coord}
 z=x^2,
 \qquad
 u=\frac{z-1}{\varepsilon},
 \qquad
 Y=\frac{y}{\varepsilon^{3/2}},
\end{equation}
one obtains the fixed elliptic curve
\[
 T:\quad Y^2=u(u-A)(u-B).
\]
The surface $X_{a,b}$ is a double cover of
$T$, obtained by cross-gluing two copies of $T$ along a slit whose
endpoints coalesce as $\varepsilon\to0$.  In the limit the two copies
meet at one separating node.  We consider two metric descriptions of
this degeneration: in the first a slit shrinks, while in the second an
auxiliary flat cylinder stretches.


\subsubsection{The shrinking-slit metric}

Fix $c>0$, choose a branch of $\varepsilon^{1/2}$ along the
degeneration path, and set $c_1=\varepsilon^{1/2}c$.  On $X_{a,b}$,
consider the differential $\omega_{(0,c_1)}=2c_1x\,dx/y$.  Under the
preceding change of variables it converges to
\begin{equation}\label{eq:Talpha}
 \alpha=c\frac{du}{Y}
\end{equation}
on either copy of $T$.  The same slit is made in both copies.  Let
$s_\varepsilon$ denote its length in the flat metric $|\alpha|^2$.

\begin{figure}[H]
\centering
\begin{tikzpicture}[line cap=round,line join=round]
 \node[inner sep=0] at (-3.45,0)
  {\includegraphics[width=6.15cm]{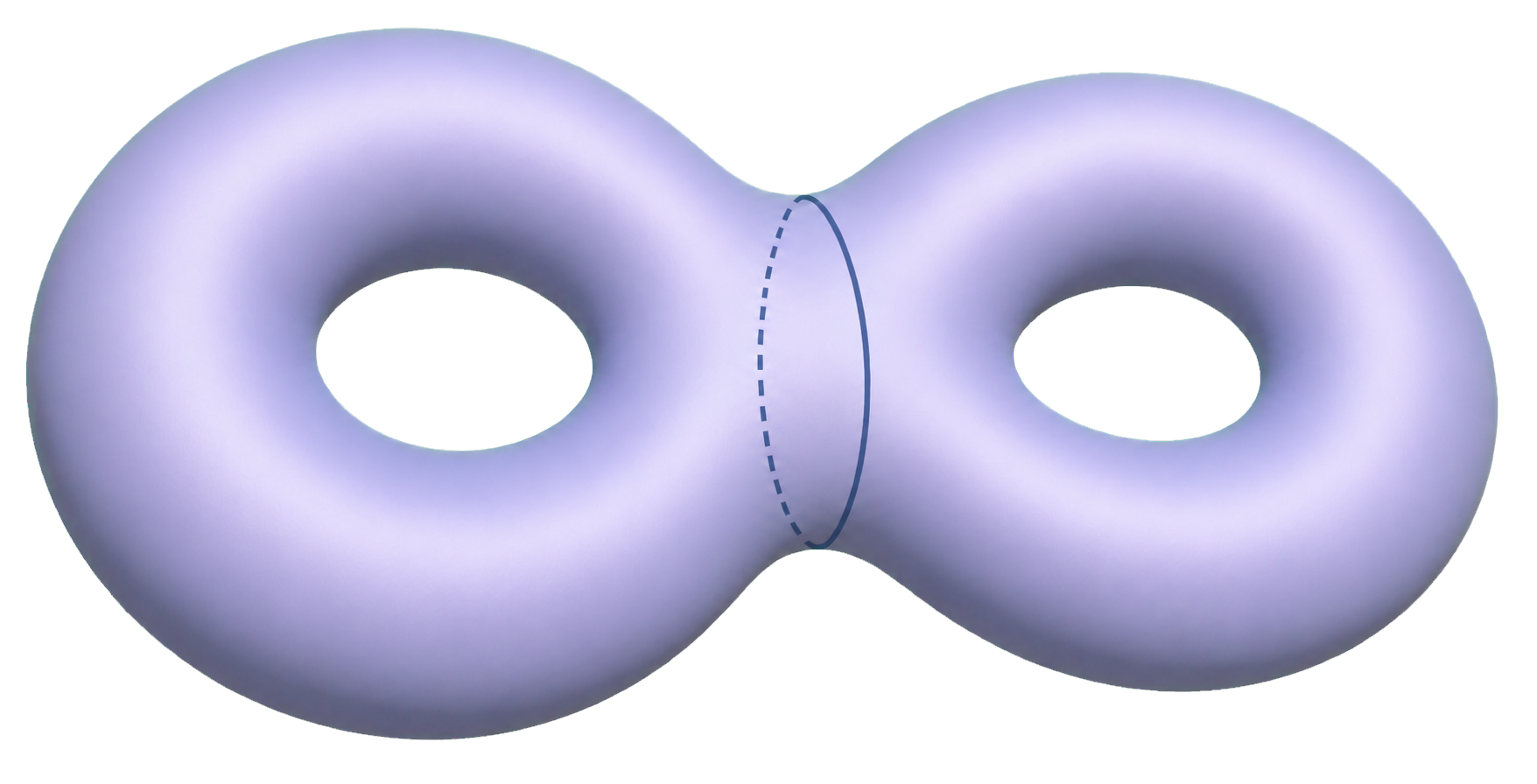}};
 \node[inner sep=0] at (3.45,0)
  {\includegraphics[width=6.05cm]{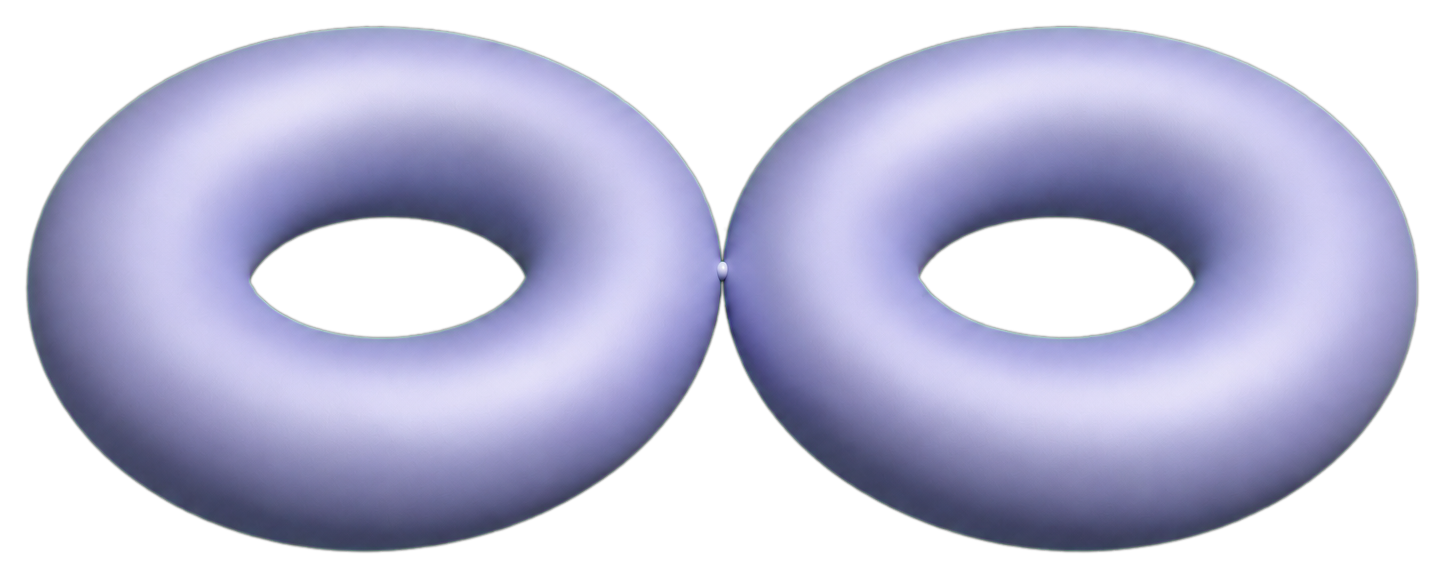}};
 \draw[-{Latex[length=2.1mm]},thick] (-.38,0)--(.38,0);
 \node[below] at (-3.45,-1.75)
  {$\bigl(X_{a,b},|\omega_{(0,c_1)}|^2\bigr)$};
 \node[below] at (3.45,-1.75)
  {$\bigl(T,|\alpha|^2\bigr)\cup\bigl(T,|\alpha|^2\bigr)$};
\end{tikzpicture}
\caption{Separating limit: two flat tori joined at a node.}
\label{fig:separating-degeneration}
\end{figure}
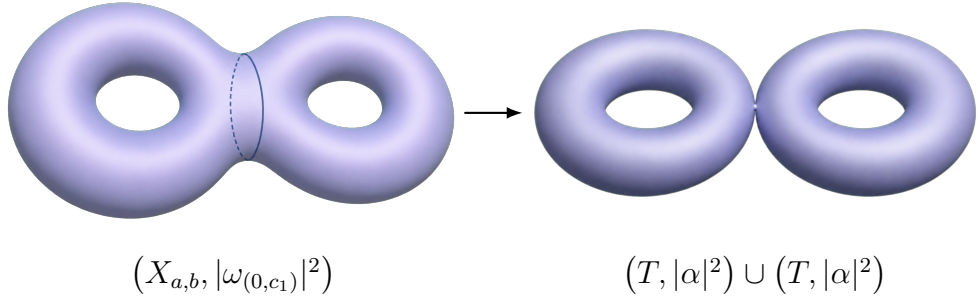

\begin{proposition}\label{prop:separating-asymptotics}
Let $A,B\in\CC\setminus\{0\}$, $A\ne B$, and $c>0$.  Choose a branch of
$\varepsilon^{1/2}$, define $a$ and $b$ as in~\eqref{eq:abc}, and set
$c_1=\varepsilon^{1/2}c$.  Let $T$ be the elliptic curve above and let
$\alpha$ be the differential in~\eqref{eq:Talpha}.
Let $s_\varepsilon$ stand for the length in the metric $|\alpha|^2$ of the shrinking
slit along which two copies of $T$ are cross-glued to form $X_{a,b}$.
Then, as $\varepsilon\to0$, we have
\[
 s_\varepsilon
 =4c\sqrt{|\varepsilon|}\,(1+o(1)),
\]
and
\begin{equation}\label{eq:separating-determinant-asymptotic}
 \frac{\Det\Delta_{|\omega_{(0,c_1)}|^2}}
 {\Area(X_{a,b},|\omega_{(0,c_1)}|^2)}
 =2^{2/3}\pi e^{6\zeta_R'(-1)}
 s_\varepsilon^{1/2}
 \left(
 \frac{\Det\Delta_{|\alpha|^2}}
 {\Area(T,|\alpha|^2)}
 \right)^2
 \bigl(1+o(1)\bigr).
\end{equation}
\end{proposition}

\begin{proof}
For the quotient map $X_{a,b}\to T$ we have
$x^2=1+\varepsilon u$.  Thus a point $(u,Y)\in T$ generally has two
preimages, corresponding to $x=\pm\sqrt{1+\varepsilon u}$.  The two
preimages merge at the two points $r_\varepsilon^\pm\in T$ over
$u=-\varepsilon^{-1}$; these are the branch points of the cover.  The
shrinking slit is the short $|\alpha|^2$-geodesic joining them.  As
$\varepsilon\to0$, both branch points approach the unique point at
infinity of the cubic model of $T$.  In the local coordinate
$t=u^{-1/2}$ they have coordinates $t=\pm i\sqrt\varepsilon$, and
$\alpha=-2c(1+O(t^2))\,dt$.  Hence
\[
 \left|\int_{r_\varepsilon^-}^{r_\varepsilon^+}\alpha\right|
 =4c\sqrt{|\varepsilon|}\,(1+o(1)),
\]
which is the stated formula for $s_\varepsilon$.

The two parameters in~\eqref{eq:lambda-coordinates} satisfy
\[
 \lambda_1=\frac BA,
 \qquad
 \lambda_2\longrightarrow\frac{B}{B-A}
 =\frac{\lambda_1}{\lambda_1-1}.
\]
The corresponding limiting Legendre curves are isomorphic.  Hence the
Kronecker limit formula and~\eqref{eq:kronecker-lambda-factor} imply
\[
 \prod_{j=1}^2
 \left(
 |\lambda_j(\lambda_j-1)|^{1/3}
 \Area(\PP^1,m_{\lambda_j})
 \right)
 \xrightarrow{\varepsilon\to0}
 2^{14/3}\pi^4
 \left(
 \frac{\Det\Delta_{|\alpha|^2}}
 {\Area(T,|\alpha|^2)}
 \right)^2.
\]
Finally, for  the binary sextic we obtain
\[
 \cP_{a,b}(0,c_1)
 =-abc^6\varepsilon^3.
\]
Substitution of these relations into the expression for $\Det\Delta_{|\omega_{(0,c_1)}|^2}$ in~\eqref{eq:main-formula}
proves the asymptotics~\eqref{eq:separating-determinant-asymptotic}.
\end{proof}

For simplicity, Proposition~\ref{prop:separating-asymptotics} is
stated for $c_0=0$.  More generally, if
$c_0=\varepsilon^{1/2}d$ and $c_1=\varepsilon^{1/2}c$, with
$c\pm d\ne0$, the limiting differentials are
$\alpha_\pm=(c\pm d)du/Y$ and the slit lengths are
$s_\varepsilon^\pm=4|c\pm d|\sqrt{|\varepsilon|}(1+o(1))$.
Thus $s_\varepsilon^{1/2}$ in
\eqref{eq:separating-determinant-asymptotic} is replaced by
$(s_\varepsilon^+s_\varepsilon^-)^{1/4}$; the area-normalized torus
determinants remain equal because $\alpha_+$ and $\alpha_-$ differ
only by constant factors.  If $c+d=0$ or $c-d=0$, the limiting
differential vanishes identically on one component; this is a
different metric degeneration and is not covered by the preceding
asymptotic formula.

\begin{remark}[Multiplicative constants]
\label{rem:jfa-normalization}
As a by-product, Proposition~\ref{prop:separating-asymptotics}
explicitly determines, for the first time, several multiplicative constants that had
previously been expressed only through unevaluated model determinants. Let
$\kappa_0$ be the constant in
\cite[Proposition~3.9]{KalvinJFA}; let $\widetilde\kappa_0$ and
$\kappa$ be the constants in the symmetric and general separating
formulas of \cite{KokotovDegeneration}; and let $\delta_g$ be the
coefficient in \cite[(1.7)]{KokotovDegeneration}, with the Bergman
tau-function normalized as in \cite[(1.8)]{KokotovDegeneration}.
Direct comparison of Proposition~\ref{prop:separating-asymptotics}
with \cite[Proposition~4]{KokotovDegeneration} yields
$\widetilde\kappa_0=2^{1/6}\pi e^{6\zeta_R'(-1)}$.  The relations in
\cite[(1.9), (1.11), and (3.28)]{KokotovDegeneration}, together with the
flat specialization of \cite[Proposition~3.9]{KalvinJFA}, determine the
remaining constants.  Thus
\begin{equation}\label{eq:absolute-normalization-constants}
 \begin{aligned}
 \widetilde\kappa_0
 &=2^{1/6}\pi e^{6\zeta_R'(-1)},
 &
 \kappa
 &=2^{2/3}\pi e^{6\zeta_R'(-1)},\\
 \kappa_0
 &=\left(\frac2\pi\right)^{1/3}e^{6\zeta_R'(-1)},
 &
 \delta_g
 &=2^{(g-5)/3}\pi^{-(g+3)/3}
 e^{6(g-1)\zeta_R'(-1)}.
\end{aligned}
\end{equation}

Consequently, on the principal stratum in every genus, the
Kokotov--Korotkin variational determinant formula
\cite{KokotovKorotkin} becomes absolutely normalized for the first
time, with no undetermined multiplicative constant left.
As a result, equation~\eqref{eq:absolute-normalization-constants}
explicitly evaluates the multiplicative constant $\kappa_0$ left undetermined in
\cite[Proposition~3.9]{KalvinJFA}.  The determinant formula
in that proposition is therefore now absolutely normalized for every
admissible conical metric on a compact Riemann surface of genus $g>1$,
including all constant-curvature conical metrics and the variable-curvature
dilation-analytic class of \cite[Definition~2.1]{KalvinJFA}.

\end{remark}

\subsubsection{The cylindrical metric and the Bismut--Bost asymptotics }
\label{sec:separating-bismut-bost-comparison}
Below, as a consequence of Proposition~\ref{prop:separating-asymptotics}, we
obtain an explicit form of the Bismut--Bost degeneration formula
\cite[Theorem~13.7]{BismutBost}.  M\"uller--M\"uller
\cite[(7.39)]{MullerMuller} identify its leading coefficient with a
relative determinant, which is therefore evaluated explicitly as well.

The metric $|\omega_{(0,c_1)}|^2$ realizes the separating degeneration
by a shrinking slit.  We now describe the same degeneration on the
same Riemann surface by a flat metric with a stretching cylinder.  Fix $\mu\notin\{0,A,B\}$ and let
\begin{equation}\label{eq:separating-cylinder-metric}
 \mathfrak q_{\varepsilon,\mu}
 =\frac{1}{4\pi^2}\bigl(x^2-1-\varepsilon\mu\bigr)\frac{dx^2}{y^2}.
\end{equation}
The quadratic differential $\mathfrak q_{\varepsilon,\mu}$ is
holomorphic and has four simple zeros over
$x^2=1+\varepsilon\mu$.  Thus the metric  $|\mathfrak q_{\varepsilon,\mu}|$ is flat and has
four conical singularities of angle $3\pi$.

In the coordinate $u$ in~\eqref{eq:degen-coord}, we have 
\[
 \mathfrak q_{\varepsilon,\mu}
 =\frac{1}{16\pi^2}\frac{u-\mu}{1+\varepsilon u}
   \frac{du^2}{u(u-A)(u-B)}
 \xrightarrow{\varepsilon\to0}
 \mathfrak q_\mu
 =\frac{1}{16\pi^2}
   \frac{(u-\mu)\,du^2}{u(u-A)(u-B)}
\]
on either limiting torus.  The quadratic differential $\mathfrak q_\mu$ has two
simple zeros and a double pole at the point $p$ over $u=\infty$.
Near $p$ there is a cylindrical coordinate $\zeta$ such that
\[
 \mathfrak q_\mu=\frac{1}{4\pi^2}
                  \frac{d\zeta^2}{\zeta^2}.
\]
Thus each limiting end is a product half-cylinder of circumference
one.  Indeed, writing
\[
 \zeta=e^{-2\pi(s+i\theta)},
 \qquad \theta\in\mathbb R/\mathbb Z,
\]
transforms its metric into \(ds^2+d\theta^2\).

Put $t=u^{-1/2}$ and $z_\pm=t(x\pm1)/2$.  Then
$z_+z_-=\varepsilon/4$.  Choosing the cylindrical coordinates on the
two limiting copies by $\zeta_\pm=z_\pm(1+O(z_\pm^2))$ implies
$\zeta_+\zeta_-=\varepsilon(1+o(1))/4$.
Hence the neck between $|\zeta_\pm|=1$  is a product cylinder  of modulus
\[
 L_\varepsilon
 :=\frac{H_\varepsilon}{\ell_\varepsilon}
 =\frac{1}{2\pi}\log\frac{4}{|\varepsilon|}+o(1),
 \qquad \varepsilon\to0,
\]
where $H_\varepsilon$ is the length and $\ell_\varepsilon=1+o(1)$ is the circumference of the cylindrical neck.

The Bismut--Bost asymptotic formula
\cite[Theorem~13.7]{BismutBost} takes the form
\begin{equation}\label{eq:Bismut--Bost}
 \Det\Delta_{|\mathfrak q_{\varepsilon,\mu}|}
 \sim C_{\mathrm{BB}}L_\varepsilon
 \exp\left(-\frac{\pi L_\varepsilon}{3}\right).
\end{equation}
The constant $C_{\mathrm{BB}}$ is not evaluated there.  Formula
\cite[(7.39)]{MullerMuller} identifies it spectrally:
\begin{equation}\label{eq:bismut-bost-relative-constant}
 C_{\mathrm{BB}}
 =2\left(\Det_{\mathrm{rel}}
 \bigl(\Delta_{|\mathfrak q_\mu|},\Delta_0\bigr)\right)^2.
\end{equation}
Here $\Det_{\mathrm{rel}}
\bigl(\Delta_{|\mathfrak q_\mu|},\Delta_0\bigr)$ is the relative
determinant, in the sense of M\"uller~\cite{MullerRelative},
 for the Friedrichs Laplacian  $\Delta_{|\mathfrak q_\mu|}$, on either limiting torus 
$(T\setminus\{p\},|\mathfrak q_\mu|)$ with a cylindrical end,  relative to  the Dirichlet
Laplacian  $\Delta_0$ on the unit half-cylinder $0<|\zeta|\leq1$.  The two limiting components are isometric, so their relative
determinants coincide; this accounts for the square in
\eqref{eq:bismut-bost-relative-constant}.   Although the cited results of Bismut--Bost and M\"uller--M\"uller are
stated for smooth metrics, they extend to the singular metric
$|\mathfrak q_{\varepsilon,\mu}|$, as shown in
Lemma~\ref{lem:conical-analytic-surgery} at the end of this subsection.


Formula \eqref{eq:bismut-bost-relative-constant} is spectral but not
explicit.  As an immediate consequence of Theorem~\ref{thm:main}, we
obtain an explicit expression for $C_{\mathrm{BB}}$, and hence for the
relative determinant in
\eqref{eq:bismut-bost-relative-constant}.  

Indeed, as
$c_1=\varepsilon^{1/2}c$, the two metrics are related by
\[
 |\mathfrak q_{\varepsilon,\mu}|
 =e^{2\varphi_{\varepsilon,\mu}}|\omega_{(0,c_1)}|^2,
 \qquad
 \varphi_{\varepsilon,\mu}
 =\frac12\log
 \frac{|x^2-1-\varepsilon\mu|}
 {16\pi^2c^2|\varepsilon|\,|x|^2}.
\]
This allows one to compare the corresponding determinants of the Friedrichs Laplacians by using the singular anomaly formula~\cite[Corollary~1.3]{KalvinJFA}. 

Since both metrics are flat away from their conical points,  the
curvature terms in the singular anomaly formula vanish, while all its
local terms are read directly from the logarithmic behavior of
$\varphi_{\varepsilon,\mu}$.  As a result, the singular anomaly formula
implies
\[
 \begin{aligned}
 \frac{\Det\Delta_{|\mathfrak q_{\varepsilon,\mu}|}}
      {\Area(X_{a,b},|\mathfrak q_{\varepsilon,\mu}|)}
 ={}&2^{-4/9}3^{-4/9}\pi^{-5/9}
   \exp\left(\frac{14}{3}\zeta_R'(-1)\right)
   \left(\frac{\Gamma(1/3)}{\Gamma(2/3)}\right)^{4/3}
   \\
 &\times
   \frac{|(1+\varepsilon\mu)\mu(\mu-A)(\mu-B)|^{1/18}}
        {c^{1/2}|\varepsilon|^{1/12}|ab|^{1/12}}
   \frac{\Det\Delta_{|\omega_{(0,c_1)}|^2}}
        {\Area(X_{a,b},|\omega_{(0,c_1)}|^2)}.
 \end{aligned}
\]
Together with Proposition~\ref{prop:separating-asymptotics}, this
reduces $C_{\mathrm{BB}}$ to the area-normalized torus determinant.  The
latter is evaluated explicitly by the argument used in the proof of
Lemma~\ref{lem:elliptic-determinants}:

\[
 \frac{\Det\Delta_{|\alpha|^2}}
      {\Area(T,|\alpha|^2)}
 =\frac{|\lambda(\lambda-1)|^{1/3}}
 {2^{7/3}\pi^2}\Area(\PP^1,m_\lambda),
 \qquad \lambda=\frac BA.
\]
Combining the preceding formulas, we obtain
\begin{equation}\label{eq:bismut-bost-explicit-constant}
 \begin{split}
 C_{\mathrm{BB}}
 &=2^{-28/9}3^{-4/9}\pi^{-32/9}
   \exp\left(\frac{32}{3}\zeta_R'(-1)\right)
   \left(\frac{\Gamma(1/3)}{\Gamma(2/3)}\right)^{4/3}
   |\mu(\mu-A)(\mu-B)|^{1/18}
   \\
 &\hspace{16mm}\times
   |\lambda(\lambda-1)|^{2/3}
   \Area(\PP^1,m_\lambda)^2,
   \qquad \lambda=\frac BA,
 \end{split}
\end{equation}
where the last area is given explicitly by
\eqref{eq:base-area-function}.

We conclude this subsection by justifying the use of the
Bismut--Bost and M\"uller--M\"uller results for the present conical
family.

\begin{lemma}\label{lem:conical-analytic-surgery}
The Bismut--Bost asymptotics~\eqref{eq:Bismut--Bost} and the M\"uller--M\"uller
identity~\eqref{eq:bismut-bost-relative-constant}, originally stated for smooth metrics, remain valid for the
Friedrichs Laplacians of the conical metrics
$|\mathfrak q_{\varepsilon,\mu}|$.  As in the smooth case, the
coefficient $C_{\mathrm{BB}}$ depends on the geometry of the two
limiting components and satisfies~\eqref{eq:bismut-bost-relative-constant}, whereas the cylindrical factor
$L_\varepsilon e^{-\pi L_\varepsilon/3}$ is universal.
\end{lemma}

\begin{proof}
 In a neighborhood of each conical point write $ |\mathfrak q_{\varepsilon,\mu}|=|x|\,|dx|^2$ in a local coordinate $x$.
Fix a small $\delta>0$ so that the four coordinate disks
$\{|x|\leq\delta\}$ and the neck between $|\zeta_\pm|=1$ are pairwise disjoint.  Let $\chi_\delta$ be a
smooth cutoff which vanishes near $x=0$ and equals one near the
boundary of the disk, and set $\psi:=\frac12\chi_\delta\log|x|$. Replacing $|\mathfrak q_{\varepsilon,\mu}|$ by $e^{2\psi}|dx|^2$ in
these disks, and leaving it unchanged elsewhere, defines a smooth metric
$m_0$.  Thus
\[
 |\mathfrak q_{\varepsilon,\mu}|=e^{2\varphi}m_0,
 \qquad
 \varphi=\frac12(1-\chi_\delta)\log|x|
\]
in the $\delta$-disks and $\varphi=0$ elsewhere.

For the smooth family $m_0$, the Bismut--Bost formula
\cite[Theorem~13.7]{BismutBost} reads
\begin{equation}\label{eq:BB-asymptotics}
 \Det\Delta_{m_0}
 = C_{\mathrm{BB},\delta} 
 L_\varepsilon e^{-\pi L_\varepsilon/3}(1+o(1)),\quad \varepsilon\to 0.
\end{equation}

On the other hand, $|\mathfrak q_{\varepsilon,\mu}|$ is flat, and $\psi(0)=0$.  Hence the singular
anomaly formula \cite[Theorem~1.1(1)]{KalvinJFA} implies
$$
 \log
 \frac{\Det\Delta_{|\mathfrak q_{\varepsilon,\mu}|}/
       \Area(X_{a,b},|\mathfrak q_{\varepsilon,\mu}|)}
      {\Det\Delta_{m_0}/\Area(X_{a,b},m_0)}
 =-\frac4{12\pi}
 \int_{|x|\leq\delta}K_0\varphi\,dA_0 
 -4C\!\left(\frac12\right)
 =:\mathcal A_\delta,
$$
where $K_0$ is the Gaussian curvature of $m_0$ and 
$\mathcal A_\delta$ is independent of $\varepsilon$. Moreover, the areas satisfy
\[
 \frac{\Area(X_{a,b},|\mathfrak q_{\varepsilon,\mu}|)}
      {\Area(X_{a,b},m_0)}\longrightarrow1,\qquad \varepsilon\to 0.
\]

This together with~\eqref{eq:BB-asymptotics} implies 
\[
 \Det\Delta_{|\mathfrak q_{\varepsilon,\mu}|}
 =C_{\mathrm{BB}}
 L_\varepsilon e^{-\pi L_\varepsilon/3}(1+o(1)),
 \qquad
 C_{\mathrm{BB}}
 =e^{\mathcal A_\delta}
  C_{\mathrm{BB},\delta}.
\]
Applying the anomaly formula with two admissible radii
$\delta_1$ and $\delta_2$ (possibly with completetly different smoothings), and using that both smoothings leave
$L_\varepsilon$ unchanged and that their areas have ratio tending to
one, gives
\[
 e^{\mathcal A_{\delta_1}}C_{\mathrm{BB},\delta_1}
 =e^{\mathcal A_{\delta_2}}C_{\mathrm{BB},\delta_2}.
\]
This justifies the Bismut--Bost asymptotic formula~\eqref{eq:Bismut--Bost} for the singular setting.

It remains to prove \eqref{eq:bismut-bost-relative-constant}. Put
\(g=|\mathfrak q_\mu|\), and let \(g_\delta\) be obtained from \(g\)
by the same smoothing at its two conical points. Cut their common
cylindrical end at the  cross-section $|\zeta_\pm|=1$ and denote the compact part by
\(K\). The proof of the relative decomposition formula in
\cite[Theorem~1, Section~2.4, equation~(2.33) and the passage to zero following it]{HillairetKalvinKokotovCMP} carries over to the present
setting without difficulty. Since \(g\) and \(g_\delta\) are conformal and coincide
near the cut, their Neumann jump operators at zero  coincide.
Thus the decomposition formula yields
\[
\log \frac{\Det_{\mathrm{rel}}(\Delta_g,\Delta_0)}
      {\Det_{\mathrm{rel}}(\Delta_{g_\delta},\Delta_0)}
 =
\log  \frac{\Det\Delta^D_{(K,g)}}
      {\Det\Delta^D_{(K,g_\delta)}},
\]
where $\Delta^D_{(K,g)}$ and $\Delta^D_{(K,g_\delta)}$ are the Friedrichs Dirichlet Laplacians on $K$.
The singular anomaly formula
\cite[Theorem~1.1(2)]{KalvinJFA} identifies the right hand side  with
\(\frac 1 2 \mathcal A_\delta\), since each limiting component
contains only two of the four conical points. Therefore
\[
 \Det_{\mathrm{rel}}(\Delta_g,\Delta_0)
 =e^{\frac 1 2 \mathcal A_\delta}
  \Det_{\mathrm{rel}}(\Delta_{g_\delta},\Delta_0).
\]
For two admissible radii $\delta_1$ and $\delta_2$ (possibly with different smoothings), the same identity
gives
\[
 e^{\frac12\mathcal A_{\delta_1}}
 \Det_{\mathrm{rel}}(\Delta_{g_{\delta_1}},\Delta_0)
 =e^{\frac12\mathcal A_{\delta_2}}
 \Det_{\mathrm{rel}}(\Delta_{g_{\delta_2}},\Delta_0)
 =\Det_{\mathrm{rel}}(\Delta_g,\Delta_0),
\]
so this quantity is independent of the smoothing radius.
For the smooth metric, \cite[(7.39)]{MullerMuller} gives
\[
 C_{\mathrm{BB},\delta}
 =2\bigl(\Det_{\mathrm{rel}}
   (\Delta_{g_\delta},\Delta_0)\bigr)^2.
\]
Together with \(C_{\mathrm{BB}}=e^{\mathcal A_\delta}
C_{\mathrm{BB},\delta}\), this implies
\[
 C_{\mathrm{BB}}
 =2\left(e^{\frac12\mathcal A_\delta}
   \Det_{\mathrm{rel}}(\Delta_{g_\delta},\Delta_0)\right)^2
 =2\bigl(\Det_{\mathrm{rel}}
   (\Delta_g,\Delta_0)\bigr)^2,
\]
which is \eqref{eq:bismut-bost-relative-constant}.
\end{proof}

\subsection{The one-node nonseparating degeneration}
\label{sec:nonseparating-degeneration}

Fix $b\in\CC\setminus\{0,1\}$, and let
\begin{equation}\label{eq:nonseparating-family}
 a=\varepsilon,\qquad \varepsilon\to0.
\end{equation}
At $\varepsilon=0$ the curve has one nonseparating node at
$(x,y)=(0,0)$.  Its normalization is the elliptic curve
\begin{equation}\label{eq:nonseparating-normalization}
 \begin{gathered}
 T:\quad v^2=(x^2-1)(x^2-b),\\
 (x,v)\longmapsto(x,y=xv).
 \end{gathered}
\end{equation}
The two preimages of the node are
$p_\pm=(0,\pm\sqrt b)$.  The algebraic limit is the nodal curve
obtained by identifying $p_-$ with $p_+$ on $T$.  This algebraic
degeneration has two different metric realizations within the family
$|\omega_{\boldsymbol c}|^2$.  If $c_0=0$, the limiting differential
is holomorphic at $p_\pm$, and the handle collapses at finite
distance.  If $c_0\ne0$, it has simple poles with opposite residues at
$p_\pm$, and the handle stretches into a long cylinder whose limit
has two cylindrical ends.  In contrast to the separating case, both
metric regimes arise directly from the original family $| \omega_{\boldsymbol c}|^2$; no auxiliary
metric is required.  We consider the two cases separately below.

\subsubsection{The holomorphic limit}
\label{sec:holomorphic-limit}

In this subsection we consider the case $c_0=0$ and
$c_1\in\CC\setminus\{0\}$.  Using $y=xv$ on the normalization $T$
in~\eqref{eq:nonseparating-normalization}, as
$a=\varepsilon\to0$ we obtain
\begin{equation}\label{eq:holomorphic-limit-differential}
 \omega_{(0,c_1)}\longrightarrow
 \alpha:=2c_1\frac{dx}{v}.
\end{equation}
The differential $\alpha$ is holomorphic and nonzero at $p_\pm$.
Thus the vanishing handle
collapses to the identification $p_-\sim p_+$, while the metric
completion of the normalization is the smooth flat torus
$(T,|\alpha|^2)$.  The vanishing cycle is the union of two symmetric
saddle connections of equal length.  Let $s_\varepsilon$ denote their
common length in the metric $|\omega_{(0,c_1)}|^2$.

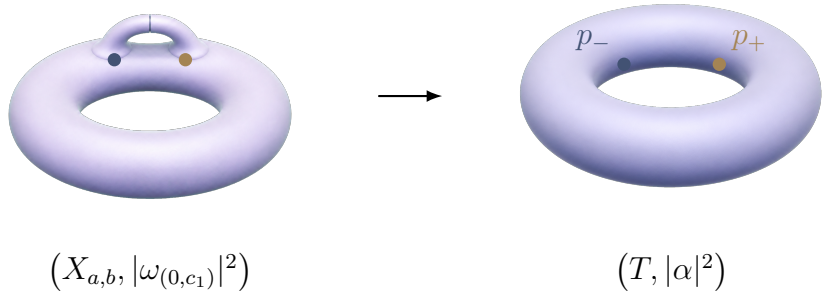
\begin{figure}[H]
\centering
\begin{tikzpicture}[line cap=round,line join=round]
 \definecolor{cycleblue}{RGB}{66,82,115}
 \definecolor{pointochre}{RGB}{164,132,84}
 \node[inner sep=0] at (-3.45,0)
  {\includegraphics[width=5.55cm]{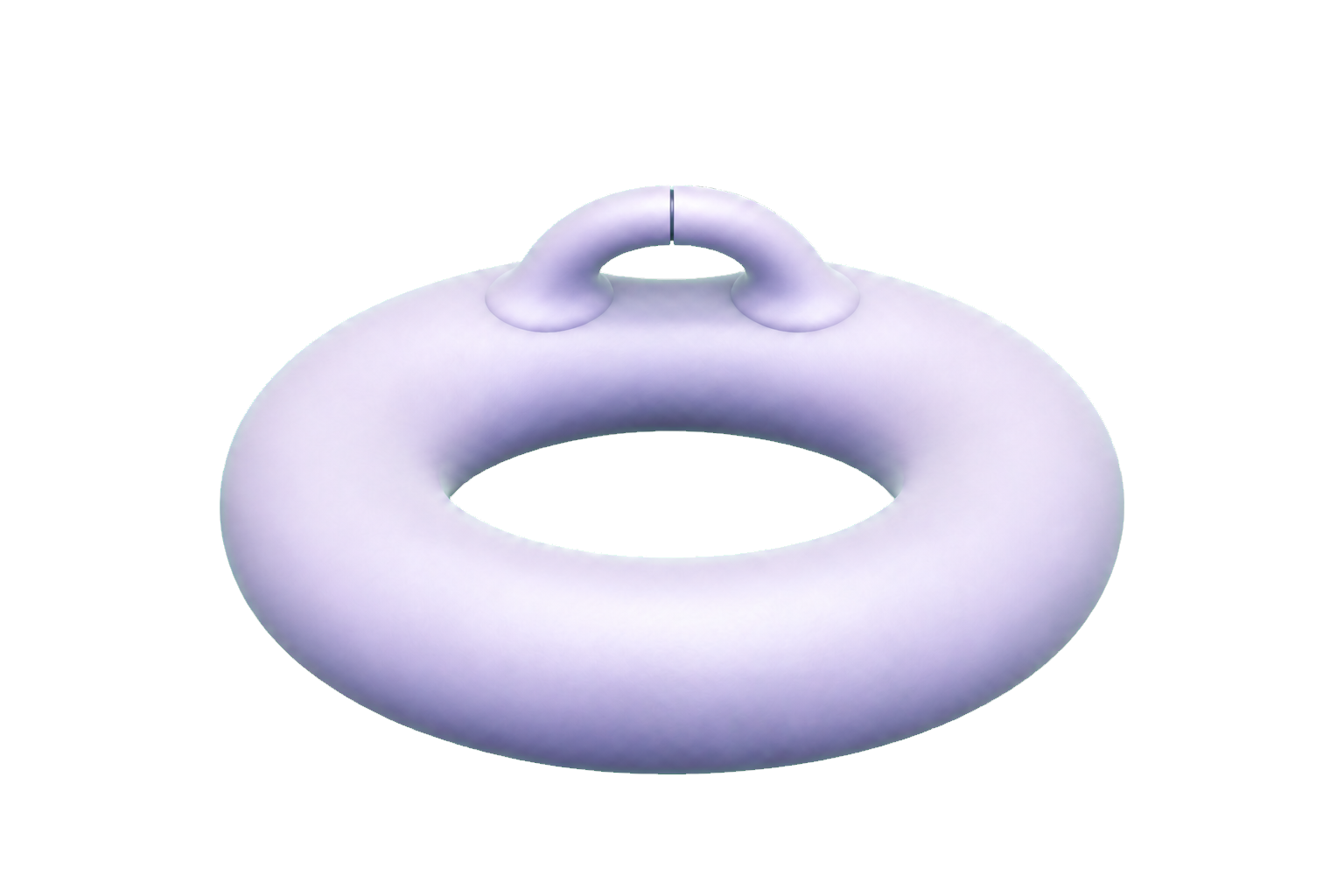}};
 \node[inner sep=0] at (3.45,0)
  {\includegraphics[width=5.35cm]{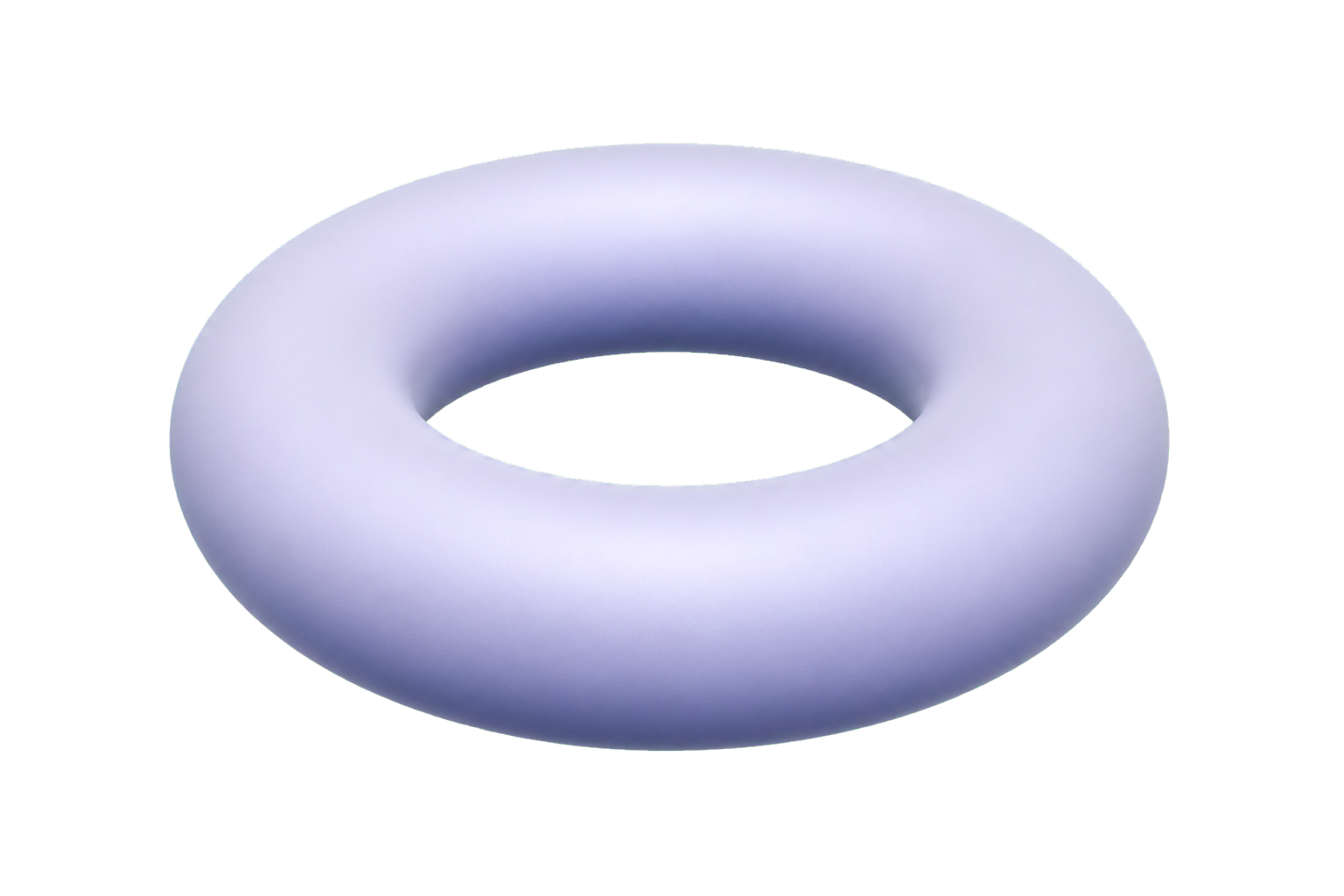}};
 \draw[-{Latex[length=2.1mm]},thick] (-.42,0)--(.42,0);
 \fill[cycleblue] (-3.92,.48) circle (2.5pt);
 \fill[pointochre] (-2.98,.48) circle (2.5pt);
 \fill[cycleblue] (2.82,.42) circle (2.5pt)
  node[above left=1pt] {$p_-$};
 \fill[pointochre] (4.08,.42) circle (2.5pt)
  node[above right=1pt] {$p_+$};
 \node[below] at (-3.45,-1.92)
  {$\bigl(X_{a,b},|\omega_{(0,c_1)}|^2\bigr)$};
 \node[below] at (3.45,-1.92) {$\bigl(T,|\alpha|^2\bigr)$};
\end{tikzpicture}
\caption{Holomorphic nonseparating limit: the marked handle collapses
to the two points $p_-$ and $p_+$ on $T$.}
\label{fig:nonseparating-holomorphic-limit}
\end{figure}

\begin{proposition}\label{prop:nonseparating-asymptotics}
Let $b,c_1\in\CC\setminus\{0\}$, with $b\ne1$, and set
$a=\varepsilon$.  Let $T$ and $\alpha$ be as in
\eqref{eq:nonseparating-normalization} and
\eqref{eq:holomorphic-limit-differential}, respectively, and let
$s_\varepsilon$ be the common length, in the metric
$|\omega_{(0,c_1)}|^2$, of the two symmetric saddle connections
forming the vanishing cycle on $X_{a,b}$.  Then, as
$a=\varepsilon\to0$, we have
\[
 s_\varepsilon
 =4|c_1|\sqrt{\frac{|\varepsilon|}{|b|}}
 \bigl(1+o(1)\bigr),
\]
and
\begin{equation}\label{eq:nonseparating-a-asymptotic}
 \frac{\Det\Delta_{|\omega_{(0,c_1)}|^2}}
 {\Area(X_{\varepsilon,b},|\omega_{(0,c_1)}|^2)}
 =2^{2/3}e^{6\zeta_R'(-1)}
 \bigl(s_\varepsilon^{1/2}|\log s_\varepsilon|\bigr)
 \frac{\Det\Delta_{|\alpha|^2}}
 {\Area(T,|\alpha|^2)}
 \bigl(1+o(1)\bigr).
\end{equation}
\end{proposition}

\begin{proof}
The differential $\omega_{(0,c_1)}$ has two zeros over $x=0$.
Each shrinking saddle connection joins these zeros through one of the
branch points $x=\pm\sqrt a$.  Traversing the two sheets through
$x=\sqrt a$, its period is
\[
4c_1\int_0^{\sqrt a}
 \frac{x\,dx}
 {\sqrt{(x^2-1)(x^2-a)(x^2-b)}}
 =
 \frac{4c_1\sqrt a}{\sqrt{-b}}
 \bigl(1+o(1)\bigr).
\]
Its modulus gives the stated formula for $s_\varepsilon$.

In order to deduce the spectral determinant asymptotics, consider the elliptic quotient
\begin{equation}\label{eq:nonseparating-elliptic-quotient}
 \begin{gathered}
 E:\quad Y^2=z(z-1)(z-b),
 \qquad
 \nu=\frac{dz}{Y},\\
 \pi:T\longrightarrow E,
 \qquad
 (x,v)\longmapsto(z=x^2,Y=xv).
 \end{gathered}
\end{equation}
The map $\pi$ is an unramified double cover, and the differentials are related as $
 \alpha=c_1\pi^*\nu$.

The Legendre parameters of $E_1$ and $E_2$ satisfy
\[
 \lambda_1\longrightarrow1-b,
 \qquad
 \lambda_2
 =-\frac{1-b}{b}\varepsilon(1+o(1)),\qquad a=\varepsilon\to 0.
\]
Accordingly, $E_1$ converges to $E$, whereas $E_2$ acquires a node, cf.~\eqref{eq:elliptic-quotients-intro}.
This together with the relation~\eqref{eq:kronecker-lambda-factor} between the Kronecker formula and the area $(\PP^1, m_\lambda )$ in~\eqref{eq:base-area-function} implies
\begin{equation}\label{eq:area-to-det}
 \begin{aligned}
 |\lambda_1(\lambda_1-1)|^{1/3}
 \Area(\PP^1,m_{\lambda_1})
 &\longrightarrow  2^{7/3}\pi^2
 \frac{\Det\Delta_{|\nu|^2}}
 {\Area(E,|\nu|^2)},\\
 |\lambda_2(\lambda_2-1)|^{1/3}
 \Area(\PP^1,m_{\lambda_2})
 &=2\pi|\lambda_2|^{1/3}
 \log\frac{16}{|\lambda_2|}(1+o(1)).
 \end{aligned}
\end{equation}
Moreover, for the binary sextic we have
\begin{equation}\label{eq:Bin-Sex-Eqn}
 \cP_{a,b}(0,c_1)=-abc_1^6, \qquad a=\varepsilon\to 0.
\end{equation}
Substituting~\eqref{eq:area-to-det} and~\eqref{eq:Bin-Sex-Eqn}  into the explicit formula~\eqref{eq:main-formula} for the area normalized determinant of $(X_{a,b}, |\omega_{(0,c_1)}|^2)$, and then using the asymptotics
for $s_\varepsilon$, we obtain
\begin{equation}\label{eq:asymp-det-deg21}
 \begin{aligned}
 \frac{\Det\Delta_{|\omega_{(0,c_1)}|^2}}
 {\Area(X_{\varepsilon,b},|\omega_{(0,c_1)}|^2)}
 ={}&2^{1/3}e^{6\zeta_R'(-1)}
 \left|\frac{(1-b)^2}{b}\right|^{1/6}
 \bigl(s_\varepsilon^{1/2}|\log s_\varepsilon|\bigr)\\
 &\times
 \frac{\Det\Delta_{|\nu|^2}}
 {\Area(E,|\nu|^2)}
 \bigl(1+o(1)\bigr),\qquad \varepsilon\to 0.
 \end{aligned}
\end{equation}

Finally, let $\tau$ be a period modulus of $E$ such that $b=\lambda(\tau)$.
For compatible period bases, the unramified double cover
$\pi:T\to E$ in~\eqref{eq:nonseparating-elliptic-quotient} has modulus
$\tau/2$.  The Kronecker limit formula and the Dedekind eta duplication formula
therefore imply
\[
 \frac{\Det\Delta_{|\alpha|^2}}{\Area(T,|\alpha|^2)}
 =2^{-1/3}\left|\frac{(1-b)^2}{b}\right|^{1/6}
 \frac{\Det\Delta_{|\nu|^2}}{\Area(E,|\nu|^2)}.
\]
This identity together with the asymptotics~\eqref{eq:asymp-det-deg21} completes the proof of~\eqref{eq:nonseparating-a-asymptotic}.
\end{proof}
\subsubsection{The meromorphic limit}
\label{sec:meromorphic-limit}

In this subsection we consider the case  $ c_0\ne0$. Fix $c_0,c_1\in\CC$ such that
\[
 (c_0^2-c_1^2)(c_0^2-bc_1^2)\ne0.
\]
This is precisely the condition that  \(\lim_{a\to0}\mathcal P_{a,b}(c_0,c_1)\neq 0\).
As $a=\varepsilon\to0$  the differential $\omega_{\boldsymbol c}$ 
converges to the differential
\[
 \alpha:=2(c_0+c_1x)\frac{dx}{xv}
\]
on the punctured normalization
$$T^\circ=T\setminus\{p_-,p_+\}.$$ 
The residues of $\alpha$ at $p_\pm$ are
$\pm2c_0/\sqrt b$. Therefore the metric
$|\alpha|^2$ has a cylindrical end at each of these points,  of circumference
\begin{equation}\label{eq:meromorphic-cylinder-circumference}
 \ell:=\frac{4\pi|c_0|}{\sqrt{|b|}}.
\end{equation}
As $\epsilon \to 0 $ the shrinking complex handle of $X_{\varepsilon, b}$ becomes metrically a long flat
cylinder and breaks into two cylindrical ends.  The length of the long cylinder is asymptotic to
\begin{equation}\label{eq:meromorphic-cylinder-length}
 H_\varepsilon
 :=\frac{2|c_0|}{\sqrt{|b|}}
 \log\frac{16|b|}{|\varepsilon||1-b|}\to\infty\quad\text{as}\quad \varepsilon\to 0,
\end{equation}
and its  circumference $\ell_\varepsilon$ satisfies
$\ell_\varepsilon=\ell(1+o(1))$;

 The difeerential $\alpha$ also has two simple zeros: the two points over
$x=-c_0/c_1$ if $c_1\ne0$, and the two points at infinity if $c_1=0$.
Thus  the  limiting flat metric  $|\alpha|^2$  has two cylindrical ends and two  conical singularities of order $1$,
i.e. of angle $4\pi$.    
\begin{figure}[H]
\centering
\begin{tikzpicture}[line cap=round,line join=round]
 \definecolor{cycleblue}{RGB}{66,82,115}
 \node[inner sep=0] at (-3.45,0)
  {\includegraphics[width=5.35cm]{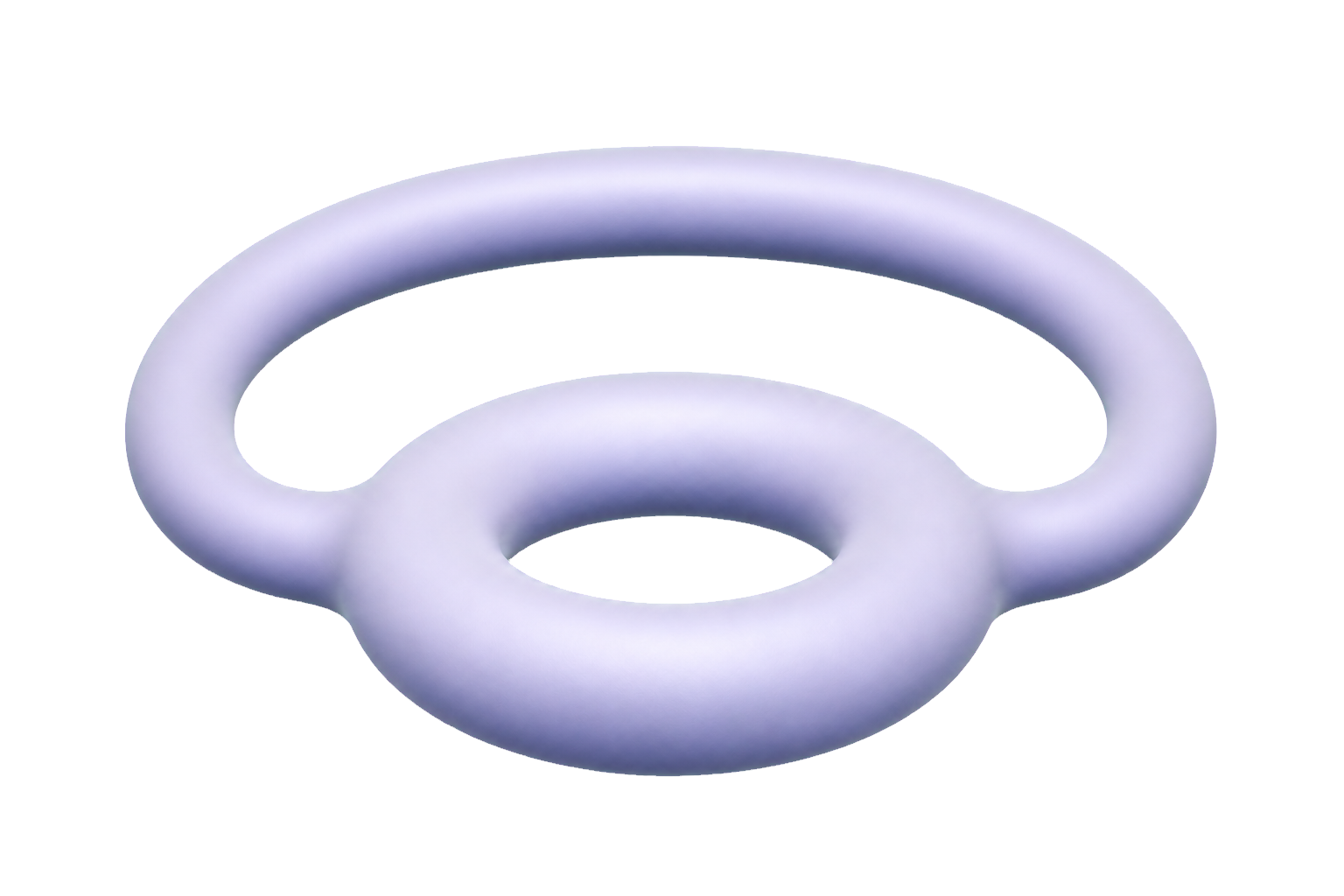}};
 \node[inner sep=0] at (3.45,0)
  {\includegraphics[width=5.65cm]{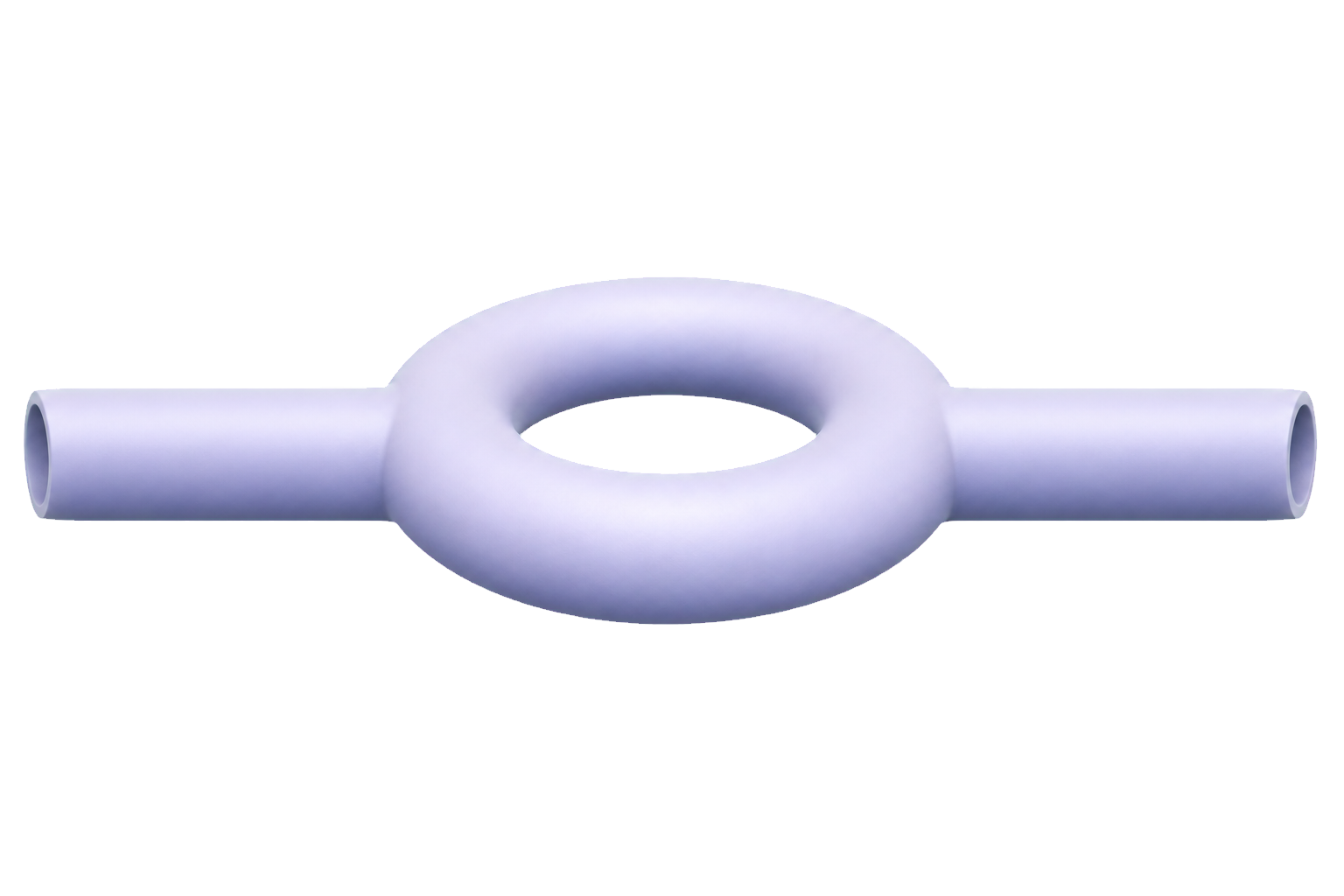}};
 \draw[-{Latex[length=2.1mm]},thick] (-.40,0)--(.40,0);
 \draw[cycleblue,line width=1.25pt]
  (-3.45,.775) arc[start angle=-90,end angle=90,
                    x radius=.10cm,y radius=.21cm];
 \draw[cycleblue,densely dashed,line width=1.1pt]
  (-3.45,1.195) arc[start angle=90,end angle=270,
                     x radius=.10cm,y radius=.21cm];
 \node[below] at (-3.45,-1.92)
 {$\bigl(X_{a,b},|\omega_{\boldsymbol c}|^2\bigr)$};
 \node[below,font=\footnotesize] at (3.45,-1.92)
  {$\bigl(T^\circ,|\alpha|^2\bigr)$};
\end{tikzpicture}
\caption{Meromorphic nonseparating limit: the marked handle becomes a
long flat cylinder and breaks into two cylindrical ends.}
\label{fig:nonseparating-meromorphic-limit}
\end{figure}
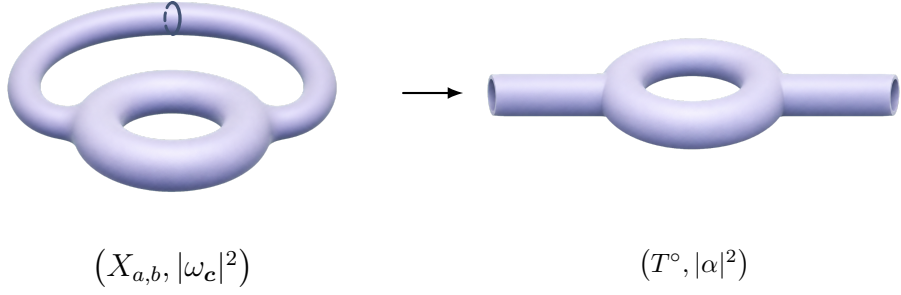

The explicit determinant formula gives the following complete
asymptotics.

\begin{proposition}\label{prop:meromorphic-nonseparating-asymptotic}
Let $b,c_0\in\CC\setminus\{0\}$, with $b\ne1$, and let $c_1\in\CC$.  Assume
that
\[
 (c_0^2-c_1^2)(c_0^2-bc_1^2)\ne0,
\]
and let $a=\varepsilon\to 0$.  The stretching cylinder has
circumference $\ell_\varepsilon=\ell(1+o(1))$, where $\ell$ is defined
in~\eqref{eq:meromorphic-cylinder-circumference}, and its length is
asymptotic to $H_\varepsilon$ defined
in~\eqref{eq:meromorphic-cylinder-length}.  Moreover, as
$\varepsilon\to0$, we have
\begin{equation}\label{eq:meromorphic-area-asymptotic}
 \Area(X_{\varepsilon,b},|\omega_{\boldsymbol c}|^2)
 =\ell H_\varepsilon\bigl(1+o(1)\bigr),
\end{equation}
and
\begin{equation}\label{eq:meromorphic-determinant-asymptotic}
\begin{split}
 \Det\Delta_{|\omega_{\boldsymbol c}|^2}
 ={}&2^{-1/3}\pi^{-1}e^{6\zeta_R'(-1)}|c_0|^{1/6}
 \left|(c_0^2-c_1^2)(c_0^2-bc_1^2)\right|^{1/12}\\
 &\times |1-b|^{1/2}\Area(\PP^1,m_{1-b})
 H_\varepsilon^2
 \exp\left(-\frac{\pi H_\varepsilon}{3\ell}\right)
 \bigl(1+o(1)\bigr).
\end{split}
\end{equation}
Here $\Area(\PP^1,m_{1-b})$ is given explicitly by
\eqref{eq:base-area-function}.
\end{proposition}

\begin{proof}
The circumference $\ell_\varepsilon$ is the modulus of the period of
$\omega_{\boldsymbol c}$ along the vanishing cycle.  After the
substitution $x=\sqrt\varepsilon\,t$, the term containing $c_1$ is
odd. Hence the period is
\[
 4c_0\int_{-1}^{1}
 \frac{dt}
 {\sqrt{(\varepsilon t^2-1)(t^2-1)(\varepsilon t^2-b)}}
 =\frac{4\pi c_0}{\sqrt{-b}}\bigl(1+o(1)\bigr).
\]
Taking the modulus gives
$\ell_\varepsilon=\ell(1+o(1))$ as stated.

The Legendre parameters in~\eqref{eq:lambda-coordinates} satisfy
\[
 \lambda_1\longrightarrow1-b,
 \qquad
 \lambda_2=-\frac{1-b}{b}\varepsilon\bigl(1+o(1)\bigr).
\]
Consequently, $H_\varepsilon$ from~\eqref{eq:meromorphic-cylinder-length}    and $\ell$ from~\eqref{eq:meromorphic-cylinder-circumference} satisfy
\[
 \exp\left(-\frac{2\pi H_\varepsilon}{\ell}\right)
 =\frac{|\lambda_2|}{16}\bigl(1+o(1)\bigr).
\]
Thanks to the explicit expression~\eqref{eq:base-area-function} for the area  $\Area(\PP^1,m_{\lambda_1})$,
we obtain 
\[
 \begin{aligned}
 |\lambda_1(\lambda_1-1)|^{1/3}
 \Area(\PP^1,m_{\lambda_1})
 &=
 |b(1-b)|^{1/3}\Area(\PP^1,m_{1-b})+ o(1),\\
 |\lambda_2(\lambda_2-1)|^{1/3}
 \Area(\PP^1,m_{\lambda_2})
 &=2\pi|\lambda_2|^{1/3}
   \log\frac{16}{|\lambda_2|}
   \bigl(1+o(1)\bigr).
 \end{aligned}
\]
In particular,
\[
 \Area(\PP^1,m_{\lambda_2})
 =\frac{\pi\sqrt{|b|}}{|c_0|}
  H_\varepsilon\bigl(1+o(1)\bigr).
\]
In the explicit formula~\eqref{eq:general-area}    for the total area $ \Area(X_{\varepsilon,b},|\omega_{\boldsymbol c}|^2)$ in Theorem~\ref{thm:main}, the divergent
contribution comes from the area  $\Area(\PP^1,m_{\lambda_2})$ of the degenerating sphere; hence  the preceding asymptotic formula proves
\eqref{eq:meromorphic-area-asymptotic}.  Since the area of the neck complement
remains bounded,  the length of the stretching cylinder is asymptotic to 
$
 \Area(X_{\varepsilon,b},|\omega_{\boldsymbol c}|^2)/ \ell,
$
or, equivalently, it is asymptotic to $H_\varepsilon$ in~\eqref{eq:meromorphic-cylinder-length}.

Furthermore, the binary sextic satisfies
\[
 \cP_{a,b}(c_0,c_1)
 =c_0^2(c_0^2-c_1^2)(c_0^2-bc_1^2)(1+o(1)),
\]
whereas the definition of $H_\varepsilon$ and the value of $\ell$
imply
\[
 |\varepsilon|^{1/6}
 =2^{2/3}\frac{|b|^{1/6}}{|1-b|^{1/6}}
 \exp\left(-\frac{\pi H_\varepsilon}{3\ell}\right).
\]
  Inserting the asymptotics of
the two spherical factors and of the binary sextic into the
determinant formula of Theorem~\ref{thm:main}, and then using the area
asymptotics, we obtain
\eqref{eq:meromorphic-determinant-asymptotic}.
\end{proof}

In~\cite[Theorem~13.7]{BismutBost}, Bismut and Bost consider a smooth
surface degenerating by stretching a flat product cylinder along a
nonseparating cycle.
If $L$ denotes the cylinder length divided by its circumference, they
prove
\[
 \Det\Delta_L
 =C_{\mathrm{BB}}L^2e^{-\pi L/3}\bigl(1+o(1)\bigr),
\]
where the constant $C_{\mathrm{BB}}$ is not determined.  To compare this formula
with Proposition~\ref{prop:meromorphic-nonseparating-asymptotic}, we
introduce the cylinder parameter 
$
 L_\varepsilon={H_\varepsilon}/{\ell}$.
Then the asymptotics~\eqref{eq:meromorphic-determinant-asymptotic} for $\Det\Delta_{|\omega_{\boldsymbol c}|^2}$  determines the
corresponding Bismut--Bost coefficient:
\[
 C_{\mathrm{BB}}
 =2^{11/3}\pi e^{6\zeta_R'(-1)}|c_0|^{13/6}
 \left|(c_0^2-c_1^2)(c_0^2-bc_1^2)\right|^{1/12}
 \frac{|1-b|^{1/2}}{|b|}
 \Area(\PP^1,m_{1-b}).
\]
Thus Proposition~\ref{prop:meromorphic-nonseparating-asymptotic}
recovers the Bismut--Bost asymptotics for the present conical
family and evaluates its coefficient.
Although their theorem is originally
stated for smooth metrics, the result remains valid for the present
conical family by the same argument as in the proof of
Lemma~\ref{lem:conical-analytic-surgery}.

\subsection{The simultaneous two-node degeneration}
\label{sec:two-node-degeneration}

Fix $b\in\CC\setminus\{0,1\}$, let
\[
 a=1+\varepsilon,
 \qquad
 \varepsilon\to0,
\]
and consider the differential $\omega_{\boldsymbol c}$ on $X_{a,b}$,
where
\[
 (c_0^2-c_1^2)(c_0^2-bc_1^2)\ne0.
\]
The limiting curve has two nonseparating nodes, at $x=1$ and $x=-1$.
Its normalization is the rational curve
\[
 R:\quad v^2=x^2-b,
 \qquad
 (x,v)\longmapsto\bigl(x,y=(x^2-1)v\bigr).
\]
The two preimages of each node are identified independently.  Near
either node, $u=x^2-1$ puts the local equation into the form
$\widetilde y^{\,2}=u(u-\varepsilon)$; thus both nodes are smoothed by the
same parameter $\varepsilon$.  This equality of the two smoothing
parameters is the algebraic restriction imposed by the bielliptic
involution.

On $R$ the differential $\omega_{\boldsymbol c}$ converges to
\[
\alpha:= 2(c_0+c_1x)\frac{dx}{(x^2-1)v}.
\]
The limiting differential has simple poles at the four preimages of
the nodes.  The smooth part of the limiting nodal curve is the punctured
normalization $R^\circ$ obtained by removing these four points from
$R$.  The differential also has two simple zeros: the two points over
$x=-c_0/c_1$ if $c_1\ne0$, and the two points at infinity if $c_1=0$.
The two ends
over $x=1$ have circumference
$2\pi|c_0+c_1|/\sqrt{|1-b|}$, and the two ends over $x=-1$ have
circumference $2\pi|c_0-c_1|/\sqrt{|1-b|}$.  The metric limit is
therefore a complete flat sphere with four cylindrical ends and two
conical singularities of order $1$, i.e. of angle $4\pi$.

\begin{figure}[H]
\centering
\begin{tikzpicture}[line cap=round,line join=round]
 \definecolor{cycleblue}{RGB}{66,82,115}
 \node[inner sep=0] at (-3.45,0)
  {\includegraphics[width=5.65cm]{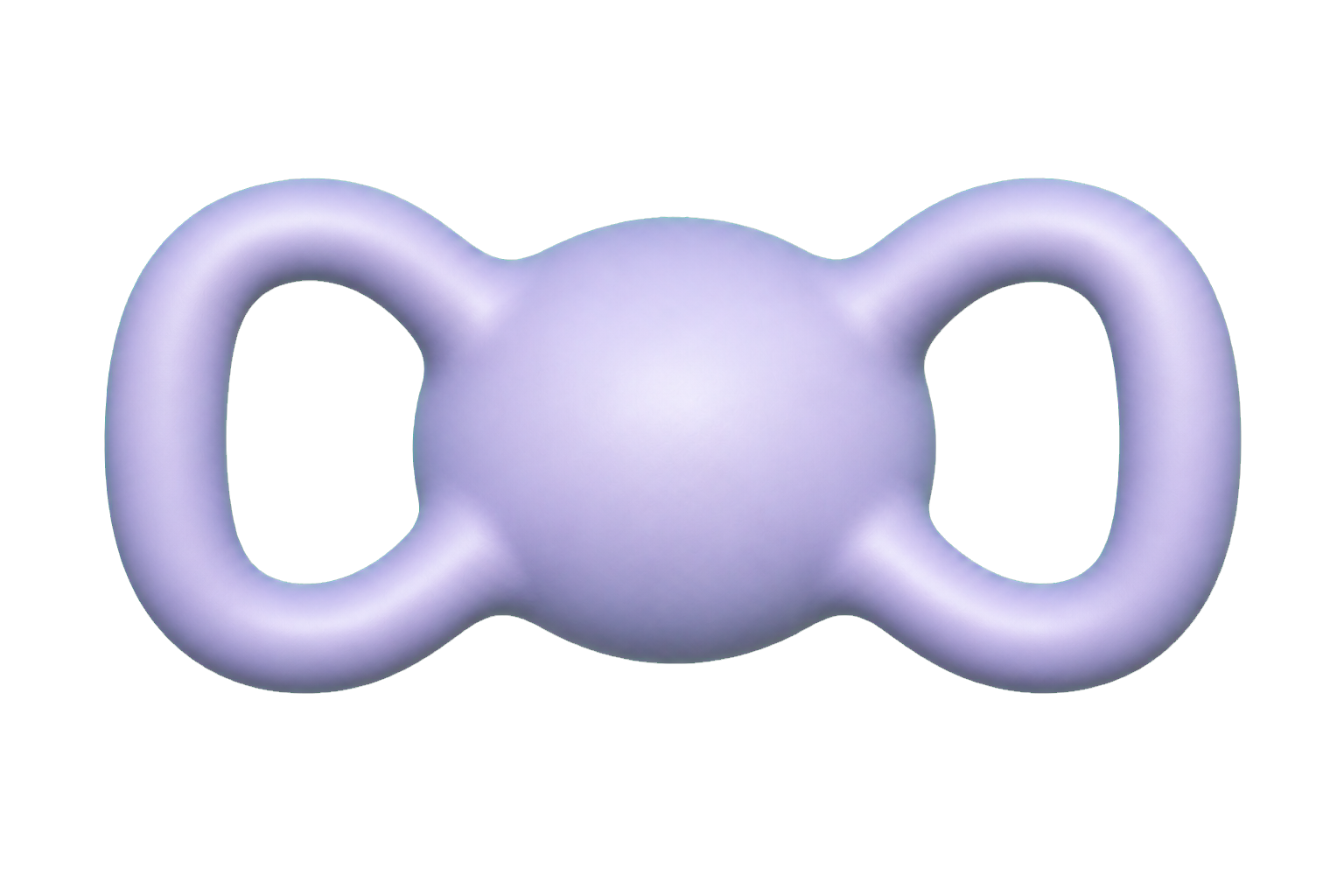}};
 \node[inner sep=0] at (3.45,0)
  {\includegraphics[width=5.65cm]{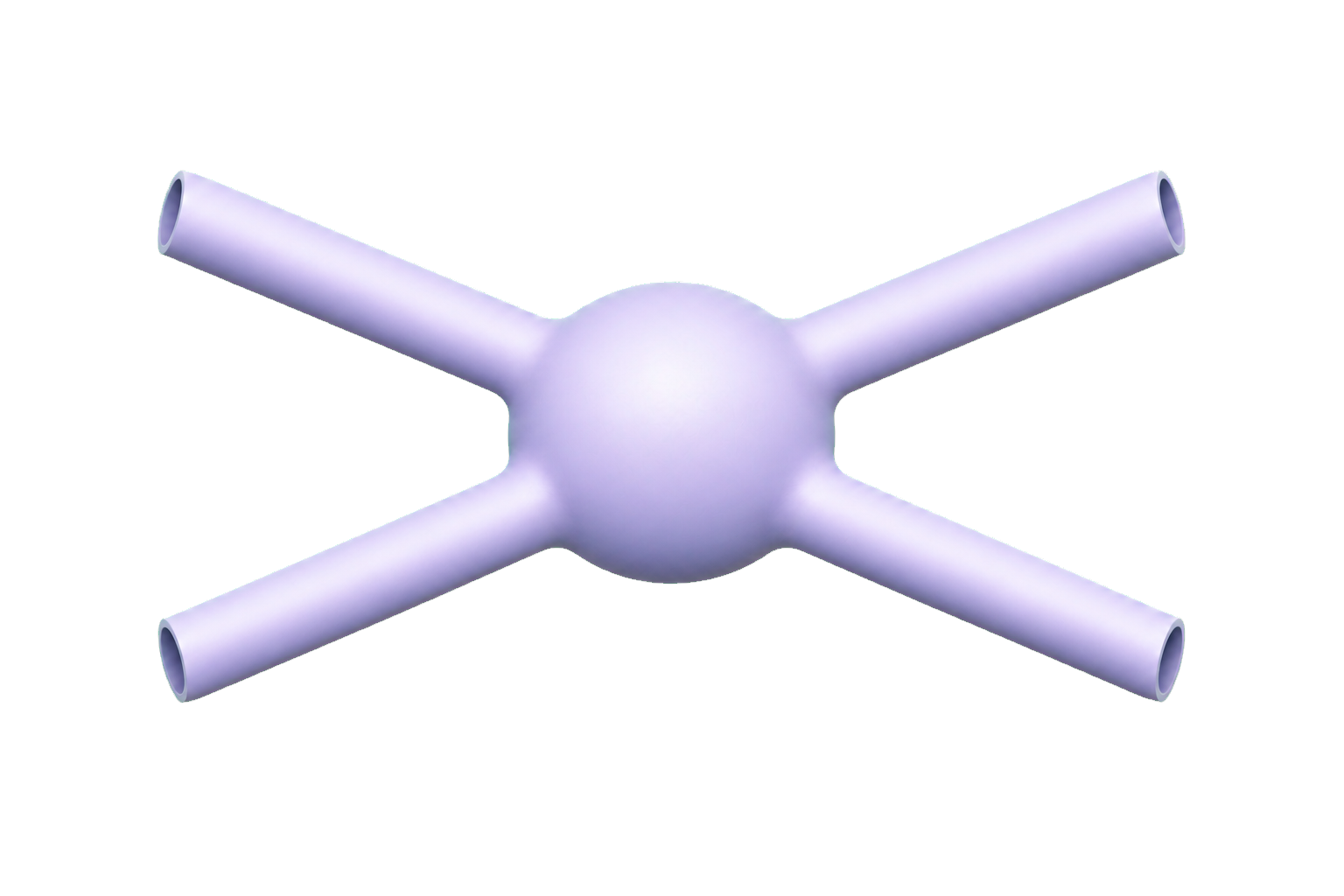}};
 \draw[-{Latex[length=2.1mm]},thick] (-.40,0)--(.40,0);
 \foreach \x in {-5.58,-1.32}{
  \draw[cycleblue,densely dashed,line width=1.05pt]
   ({\x-.27},0) arc[start angle=180,end angle=0,
                     x radius=.27cm,y radius=.095cm];
  \draw[cycleblue,line width=1.2pt]
   ({\x+.27},0) arc[start angle=0,end angle=-180,
                     x radius=.27cm,y radius=.095cm];
 }
 \node[below] at (-3.45,-1.70)
  {$\bigl(X_{a,b},|\omega_{\boldsymbol c}|^2\bigr)$};
 \node[below,font=\footnotesize] at (3.45,-1.70)
  {$\displaystyle\left(R^\circ,
    |\alpha|^2\right)$};
\end{tikzpicture}
\caption{Simultaneous two-node degeneration: two cycles collapse and
the normalization $R^\circ$ has four cylindrical ends.}
\label{fig:two-node-degeneration}
\end{figure}
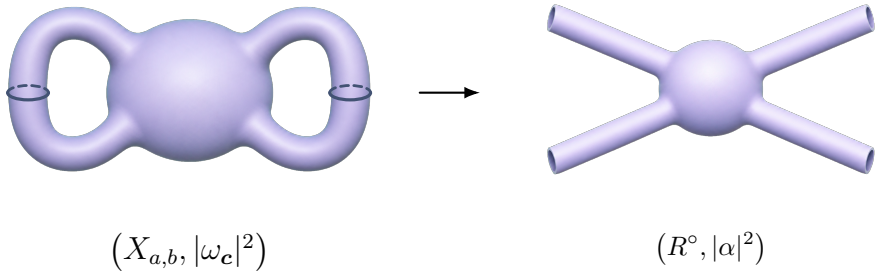

\begin{proposition}\label{prop:two-node-asymptotic}
Set
\[
 L_1(\varepsilon)=\log\frac{16|1-b|}{|\varepsilon|},
 \qquad
 L_2(\varepsilon)=\log\frac{16|a-b|}{|b\varepsilon|},
 \qquad a=1+\varepsilon,\quad\text{as}\quad\varepsilon\to0.
\]
Under the degeneration described above, the two nonseparating necks
stretch into long flat cylinders and break.  Away from the necks, the
metric converges to the complete flat metric $|\alpha|^2$ on the
four-ended sphere $R^\circ$.  The determinant tends to zero, whereas the area diverges
logarithmically.  More precisely, as $\varepsilon\to0$, we have
\begin{equation}\label{eq:two-node-determinant-asymptotic}
 \begin{split}
 \Det\Delta_{|\omega_{\boldsymbol c}|^2}
 ={}&4e^{6\zeta_R'(-1)}
 \frac{|b|^{1/6}
 |c_0^2-c_1^2|^{1/6}
 |c_0^2-bc_1^2|^{1/12}}
 {|1-b|^{5/3}}
 |\varepsilon|^{2/3}L_1(\varepsilon)L_2(\varepsilon)\\
 &\times
 \left(
  |c_1|^2L_1(\varepsilon)+|c_0|^2L_2(\varepsilon)
 \right)
 \bigl(1+o(1)\bigr),
\end{split}
\end{equation}
and
\begin{equation}\label{eq:two-node-area-asymptotic}
 \Area(X_{1+\varepsilon,b},|\omega_{\boldsymbol c}|^2)
 =\frac{8\pi}{|1-b|}
 \left(
  |c_1|^2L_1(\varepsilon)+|c_0|^2L_2(\varepsilon)
 \right)
 \bigl(1+o(1)\bigr).
\end{equation}
\end{proposition}

\begin{proof}
The two parameters $\lambda_1$ and $\lambda_2$ in~\eqref{eq:lambda-coordinates} approach
different cusps.  Local coordinates at these cusps are
\[
 \mu_1=\lambda_1^{-1}=\frac{\varepsilon}{b-1},
 \qquad
 \mu_2=1-\lambda_2=\frac{b\varepsilon}{a-b}.
\]
Explicit expression~\eqref{eq:base-area-function} for the area of $(\PP^1,m_\lambda)$ gives
\[
 \begin{aligned}
 \Area(\PP^1,m_{\lambda_1})
 &=2\pi|\mu_1|L_1(\varepsilon)(1+o(1)),\\
 \Area(\PP^1,m_{\lambda_2})
 &=2\pi L_2(\varepsilon)(1+o(1)),\\
 |\lambda_1(\lambda_1-1)|^{1/3}
 \Area(\PP^1,m_{\lambda_1})
 &=2\pi|\mu_1|^{1/3}L_1(\varepsilon)(1+o(1)),\\
 |\lambda_2(\lambda_2-1)|^{1/3}
 \Area(\PP^1,m_{\lambda_2})
 &=2\pi|\mu_2|^{1/3}L_2(\varepsilon)(1+o(1)).
 \end{aligned}
\]
Together with
\[
 \cP_{1+\varepsilon,b}(c_0,c_1)
 \longrightarrow
 (c_0^2-c_1^2)^2(c_0^2-bc_1^2),
\]
direct substitution into the explicit formula~\eqref{eq:main-formula} for  $\Det\Delta_{|\omega_{\boldsymbol c}|^2}$ and the explicit formula~\eqref{eq:general-area} for  $ \Area(X_{1+\varepsilon,b},|\omega_{\boldsymbol c}|^2)$,  proves
\eqref{eq:two-node-determinant-asymptotic} and
\eqref{eq:two-node-area-asymptotic}, respectively.
\end{proof}

\begin{remark}[Comparison with Bismut--Bost]
Bismut and Bost do not state the two-node scalar asymptotic formula
separately.  The contribution of an arbitrary number of nodes to the
Quillen metric is given in
\cite[Theorem~2.2 and Corollary~2.3, pp.~8--9]{BismutBost}, while the
passage from the Quillen metric to the scalar determinant in the
one-node case is carried out in
\cite[Theorem~13.2 and Propositions~13.3--13.5,
pp.~94--99]{BismutBost}.  With $u=x^2-1$ and
$\widetilde y=y/\sqrt{x^2-b}$, the local equation at $\varepsilon=0$ is
\[
 \widetilde y^{\,2}=u^2,
\]
which describes two branches crossing at a node.  For $\varepsilon\ne0$, it
becomes
\[
 \widetilde y^{\,2}=u(u-\varepsilon),
\]
replacing the crossing by a neck.  After centering $u$ at
$\varepsilon/2$, it becomes the standard plumbing model with parameter
$-\varepsilon^2/4$.  Since the same $\varepsilon$ occurs at $x=1$ and $x=-1$,
the two necks open simultaneously with the same plumbing parameter.
The factor $|\varepsilon|^{2/3}$ follows from
\cite[Corollary~2.3]{BismutBost}, while the $L^2$ argument in
\cite[Propositions~13.3--13.5]{BismutBost} accounts for three logarithmic factors:
one from the constant function and one from each of the two
degenerating holomorphic differentials.  Since
$L_j(\varepsilon)=\log(1/|\varepsilon|)+O(1)$, $j=1,2$, the logarithmic factor
in the asymptotic formula of
Proposition~\ref{prop:two-node-asymptotic} satisfies
\[
 L_1(\varepsilon)L_2(\varepsilon)
 \bigl(|c_1|^2L_1(\varepsilon)+|c_0|^2L_2(\varepsilon)\bigr)
 =(|c_0|^2+|c_1|^2)\log^3\frac1{|\varepsilon|}
 \bigl(1+o(1)\bigr).
\]
Hence
\[
 \Det\Delta_{|\omega_{\boldsymbol c}|^2}
 =C_{\mathrm{BB}}|\varepsilon|^{2/3}
 \log^3\frac1{|\varepsilon|}\bigl(1+o(1)\bigr),
\]
where $C_{\mathrm{BB}}$ is not evaluated in \cite{BismutBost}.
Proposition~\ref{prop:two-node-asymptotic} evaluates the coefficient
\[
 C_{\mathrm{BB}}
 =4e^{6\zeta_R'(-1)}
 \frac{|b|^{1/6}|c_0^2-c_1^2|^{1/6}
 |c_0^2-bc_1^2|^{1/12}}{|1-b|^{5/3}}\times (|c_0|^2+|c_1|^2).
\]
As before, the assumption that the metrics are smooth is removed by
the same argument as in the proof of
Lemma~\ref{lem:conical-analytic-surgery}.

We also note that although \cite{MullerMuller} contains no result
directly applicable to either the one-node nonseparating degeneration
or the simultaneous two-node degeneration, its 
analytic-surgery scheme for manifolds with boundary, combined with BFK self-gluing, can be adapted in both cases to express the Bismut--Bost constants through relative determinants of the corresponding limiting surfaces.  We do not pursue this here.

\end{remark}

\end{document}